\pdfoutput=1
\documentclass[english]{jnsao}
\usepackage[utf8]{inputenc}
\usepackage{cite}

\usepackage[nameinlink,capitalize]{cleveref}
\newcommand{\al}{\alpha}

\newcommand{\ga}{\gamma}
\newcommand{\la}{\lambda}
\newcommand{\de}{\delta}
\newcommand{\eps}{\varepsilon}
\newcommand{\bx}{\bar x}
\newcommand{\by}{\bar y}

\newcommand{\bp}{\bar p}
\newcommand{\iv}{^{-1} }

\newcommand {\R} {\mathbb R}
\newcommand {\N} {\mathbb N}

\newcommand {\B} {\mathbb B}

\newcommand {\gph} {{\textrm{gph}}\,}%Graph
\newcommand {\dom} {{\textrm{dom}}\,}
\newcommand {\epi} {{\textrm{epi}}\,}

\newcommand {\cl} {{\textrm{cl}}\,}

\newcommand {\sd} {\partial}
\newcommand{\folgt}{$ \Rightarrow\ $}

\newcommand{\ds}{\displaystyle}
\def\nbh{neighbourhood}
\def\es{\emptyset}

\def\LHS{left-hand side}
\def\RHS{right-hand side}
\def\SVM{set-valued mapping}
\def\EVP{Ekeland variational principle}
\def\Fr{Fr\'echet}
\newcommand{\ang}[1]{\left\langle #1 \right\rangle}
\newcommand{\qdtx}[1]{\quad\mbox{#1}\quad}
\newcommand{\AND}{\quad\mbox{and}\quad}
\newcounter{mycount}

\newcommand{\AK}[1]{\todo[inline]{AK {#1}}}

\title{Uniform regularity of set-valued mappings and stability of implicit multifunctions}
\shorttitle{Uniform regularity of set-valued mappings}

\author{Nguyen Duy Cuong 
    \thanks{
    Centre for Informatics and Applied Optimization, School of Engineering, IT and Physical Sciences, Federation University, POB 663, Ballarat, Vic, 3350, Australia;
{Department of Mathematics, College of Natural Sciences, Can Tho University, Vietnam}
(\email{duynguyen@students.federation.edu.au}, \email{ndcuong@ctu.edu.vn})}
\and
Alexander Y. Kruger%
    \thanks{
    Centre for Informatics and Applied Optimization, School of Engineering, IT and Physical Sciences, Federation University, POB 663, Ballarat, Vic, 3350, Australia
(\email{a.kruger@federation.edu.au}, \orcid{0000-0002-7861-7380})}
}
\shortauthor{Cuong and Kruger}

\acknowledgements{
The research was supported by the Australian Research Council, project DP160100854.
The second author benefited from the support of the European
Union's Horizon 2020 research and innovation programme under the Marie Sk{\l}odowska--Curie Grant
Agreement No. 823731 CONMECH, and Conicyt REDES program 180032.

The authors thank the referees for their careful reading of the manuscript and constructive comments and suggestions.
}

\manuscriptlicense{CC-BY-SA 4.0}

  \manuscriptsubmitted{2020-06-25}
  \manuscriptaccepted{2021-06-15}
  \manuscriptvolume{2}
  \manuscriptnumber{6599}
  \manuscriptyear{2021}
  \manuscriptdoi{10.46298/jnsao-2021-6599}

\begin{document}
%%%%%%%%%%%%%%%%%%%%%%%%%%%%%%%%%%%%%%%%%%%%%%%

\maketitle

\begin{abstract}
We propose a
%comprehensive
unifying general
(i.e. not assuming the mapping to have any particular structure) view on the theory of regularity and clarify the relationships between the existing primal and dual quantitative sufficient and necessary conditions including their hierarchy.
We expose the typical sequence of regularity assertions, often hidden in the proofs, and
the roles of the assumptions involved in the assertions, in particular, on the underlying space: general metric, normed, Banach or Asplund.
As a consequence, we formulate primal and dual conditions for the
%conventional metric regularity and subregularity properties as well as
stability properties of solution mappings to inclusions.
\end{abstract}

%%%%%%%%%%%%%%%%%%%%%%%%%%%%%%%%%%%%%%%%%%%%%%%
\section{Introduction}\label{sect1}

Many important problems in variational analysis and optimization can be modelled by an {inclusion}
%\emph{generalized equations}
$y\in F(x)$, where $F$ is a \SVM.
The behavior of the solution set $F^{-1}(y)$ when $y$ and/or $F$ are perturbed is of special interest.
The concepts of \emph{metric regularity} and \emph{subregularity} (cf., e.g., \cite{DonRoc14,Mor06.2,Iof17}) have been the key tools when studying stability of solutions.
In the next definition, we use the names \emph{$\al-$regularity} and \emph{$\al-$subregularity}, fixing the main quantitative parameter in the conventional definitions of the properties.

\begin{definition}\label{D1.1}
Let $X$ and $Y$ be metric spaces, $F:X\rightrightarrows Y$, $(\bx,\by)\in \gph F$, and $\al>0$.
The mapping $F$ is
\begin{enumerate}
\item
$\al-$regular at $(\bx,\by)$ if there exist $\delta\in]0,+\infty]$ and $\mu\in]0,+\infty]$ such that
\begin{align}\label{D1.1-1}
\al d(x,F\iv(y))\le d(y,F(x))
\end{align}
for all $x\in B_\de(\bx)$ and $y\in B_\de(\by)$ with $d(y,F(x))<\al\mu$;

\item
$\al-$subregular at $(\bx,\by)$ if there exist $\delta\in]0,+\infty]$ and $\mu\in]0,+\infty]$ such that
\begin{align}\label{D1.1-2}
\al d(x,F\iv(\by))\le d(\by,F(x))
\end{align}
for all $x\in B_\de(\bx)$ with $d(\by,F(x))<\al\mu$.
\end{enumerate}
\end{definition}

In the above definition and throughout the paper, $B_\de(\bx)$ and $B_\de(\by)$ stand for open balls with radius $\de$ around respective points in appropriate spaces.
Note that $\de$ and $\mu$
%in the above definition
can take infinite values; thus, the definition covers local as well as global properties.
This remark applies also to the subsequent definitions.
The technical conditions $d(y,F(x))<\al\mu$ and $d(\by,F(x))<\al\mu$ can be dropped (cf. \cite{Iof00,Mor06.1}), particularly because the value $\mu=+\infty$ is allowed.
This does not affect the properties themselves, but can have an effect on the value of $\de$.

Inequalities \eqref{D1.1-1} and \eqref{D1.1-2} provide linear estimates of the distance from $x$ to the solution set of the respective inclusion via the `residual' $d(y,F(x))$ or $d(\by,F(x))$.
As commented by Dontchev and Rockafellar \cite[p.178]{DonRoc14} `in applications, the residual is typically easy to compute or estimate, whereas finding a solution might be considerably more difficult'.

Besides their importance in studying stability of solutions to inclusions, regularity type estimates are involved in constraint qualifications for optimization problems, qualification conditions in subdifferential and coderivative calculus, and convergence analysis of computational algorithms \cite{Chi10,AraDonGeo07,DonVel09,AraDonGeoVel11,AdlCibNga15, AspChaLuk16,HesLuk13,LukThaTam18,CibPreRou19}.

The name `metric regularity' was coined by Borwein in 1986 \cite{Bor86}, but the concept itself
%was used much earlier.
can be traced back to the Banach--Schauder open mapping theorem for linear operators, and its nonlinear extensions due to Lyusternik \& Graves \cite{Lyu34,Gra50} and Robinson \& Ursescu \cite{Rob76,Urs75}; see, for instance, \cite{Don96,DonRoc14,DonLewRoc03,RocWet98,Mor06.1,Iof16}
for historical comments.
Unlike the `full' regularity in part (i) of \cref{D1.1}, the weaker subregularity property in part (ii) (as well as closely related to it properties like \emph{calmness, error bounds} and \emph{weak sharp minima}) is not stable under small perturbations of the data.
It has also been well studied; see, for instance, \cite{BorZhu88,DonLewRoc03, Mor06.1,YenYaoKie08,LiMor12,ApeDurStr13,DonRoc14, Kru15,ZheNg10}.
Fortunately, the subregularity property is satisfied automatically in finite dimensions when the graph of $F$ is the union of finitely many polyhedral convex sets; cf. \cite{DonRoc14,Iof17}.
\if{
\NDC{20.12.20
Could you please have a look at \cite[Theorem~7.29]{Iof16.2}? I think the theorem states that the convexity assumption can be dropped.}
}\fi

When $y$ is not fixed and can be any point in a \nbh\ of a given point $\by$, it represents \emph{canonical perturbations} of the inclusion $\by\in F(x)$.
For some applications it can be important to allow also perturbations in the \RHS.
This leads to the need to consider parametric inclusions $\by\in F(p,x)$ (or even $y\in F(p,x)$, thus, combining the two types of perturbations), where $F$ is a \SVM\ of two variables, with (nonlinear) perturbations in the \RHS\ given by a parameter $p$ from some fixed set $P$.

Along with the mapping $F:P\times X\rightrightarrows Y$, which is our main object in this paper, given a point $p\in P$, we consider the mapping $F_p:=F(p,\cdot):X\rightrightarrows Y$.
Given a $y\in Y$, the mapping
\begin{align}\label{G}
p\mapsto G(p):=F_p\iv(y)=\{x\in X\mid y\in F(p,x)\}
\end{align}
can be interpreted as an \textit{implicit multifunction} corresponding to the parametric {inclusion} $y\in F(p,x)$.
When studying implicit multifunctions, it is common to consider
`uniform' versions of the properties in \cref{D1.1} (cf., e.g., \cite[Definition~3.1]{Iof17.1}).

\begin{definition}\label{D1.2}
Let $X$ and $Y$ be metric spaces, and $P$ be a set, $F:P\times X\rightrightarrows Y$, $\bx\in X$, $\by\in Y$, and $\al>0$.
The mapping $F$ is
\begin{enumerate}
\item
$\al-$regular in $x$ uniformly in $p$ over $P$ at $(\bx,\by)$ if there exist $\delta\in]0,+\infty]$ and $\mu\in]0,+\infty]$ such that
\begin{align}\label{D1.2-1}
\al d(x,F_p\iv(y))\le d(y,F(p,x))
\end{align}
for all $p\in P$, $x\in B_\de(\bx)$ and $y\in B_\de(\by)$ with $d(y,F(p,x))<\al\mu$;

\item
$\al-$subregular in $x$ uniformly in $p$ over $P$ at $(\bx,\by)$ if there exist $\delta\in]0,+\infty]$ and $\mu\in]0,+\infty]$ such that
\begin{align}\label{D1.2-2}
\al d(x,F_p\iv(\by))\le d(\by,F(p,x))
\end{align}
for all $p\in P$ {and} $x\in B_\de(\bx)$ with ${d(\by,F(p,x))}<\al\mu$.
\end{enumerate}
\end{definition}

If $P$ is a singleton, then the properties in \cref{D1.2} reduce to the corresponding conventional regularity properties in \cref{D1.1}.
Moreover, the subregularity property in \cref{D1.2}(ii) coincides in this case with the subregularity property of the mapping \eqref{G} considered in \cite{ChuKim16}.

\begin{remark}\label{R1.1}
\begin{enumerate}
\item
If $Y$ is a linear metric space with a shift-invariant metric, in particular, a normed space, then the property in part (i) of \cref{D1.2} reduces to the one in part (ii) with the extended parameter set $\widehat{P}:=P\times Y$ and \SVM\ $\widehat{F}((p,y),x):={F(p,x)-y}$, $((p,y),x)\in\widehat{P}\times X$, in place of $P$ and~$F$, respectively.
Moreover, in both parts of the definition, it is sufficient to consider the case $\by:=0$: the general case reduces to it by replacing $F$ with $F-\by$.
\item
Unlike \cref{D1.1}, in \cref{D1.2}
%and \ref{D1.3},
the reference point $(\bx,\by)$
%under consideration in \cref{D1.2} generally does not belong to the graph of the given set-valued mapping.
is not associated with the graph of $F$.
%In the current paper, we mainly focus on the properties in \cref{D1.2}, and we will see that this assumption is not essential when formulating necessary and sufficient conditions for the properties.
This is a technical relaxation caused by the fact that $\gph F$ is a subset of a product of three spaces $P\times X\times Y$, and at this stage there is no reference point in $P$.
\cref{D1.3} below is formulated in a more conventional way.
\item
There exist other concepts of uniform regularity in the literature.
For instance, it is not uncommon to talk about uniform regularity when inequality \eqref{D1.1-1} holds for all $(\bx,\by)$ in a compact subset of $X\times Y$ with the same parameters $\al$, $\de$ and $\mu$; cf. \cite{CibPreRou19}.
\end{enumerate}	
\end{remark}

\if{
\AK{28/12/19.
I've never heard about regularity of a mapping at a point in its image.
Have you?
Can you think of a better name?

Zheng and Ng need to be checked again.
I think they used to write about regularity of inclusions instead of mappings.
This could be exactly our case.}
\NDC{2.1.20
Me too.
I think \cref{D1.2} can be rewritten
``The mapping $F_p$ is subregular at $\bx$ uniformly in $P$ if there exist exist $\delta\in]0,+\infty]$ and $\mu\in]0,+\infty]$ such that
$\al d(x,F_p\iv(\by))\le d(\by,F(p,x))$
for all $p\in P$ and $x\in B_\de(\bx)$ with $d(\by,F(p,x))<\al\mu$''.
In this case, the mapping $G$ does not seem to play any role.
What do you think about it?
Yes, in \cite{ZheNg10}, also in \cite{Ngh14}, they called \textit{metric subregularity of generalized equations}, but they studied the conventional metric (sub-)regularity.}
}\fi

Local (in $p$) versions of the properties in \cref{D1.2} are of special interest.
They correspond to $P$ being a \nbh\ of a point $\bp$ in some metric spaces; cf., e.g., \cite{NgaTroThe13,Iof17.1}.

\begin{definition}\label{D1.3}
Let $P$, $X$ and $Y$ be metric spaces, $F:P\times X\rightrightarrows Y$, $(\bp,\bx,\by)\in\gph F$, and $\al>0$.
The mapping $F$ is
\begin{enumerate}
\item
$\al-$regular in $x$ uniformly in $p$ at $(\bp,\bx,\by)$ if there exist $\eta\in]0,+\infty]$, $\delta\in]0,+\infty]$ and $\mu\in]0,+\infty]$ such that inequality \eqref{D1.2-1} is satisfied
for all $p\in B_\eta(\bp)$, $x\in B_\de(\bx)$ and $y\in B_\de(\by)$ with $d(y,F(p,x))<\al\mu$;

\item
$\al-$subregular in $x$ uniformly in $p$ at $(\bp,\bx,\by)$ if there exist $\eta\in]0,+\infty]$, $\delta\in]0,+\infty]$ and $\mu\in]0,+\infty]$ such that
inequality \eqref{D1.2-2} is satisfied
for all $p\in B_\eta(\bp)$ and $x\in B_\de(\bx)$ with ${d(\by,F(p,x))}<\al\mu$.
\end{enumerate}
\end{definition}

We often simply say that $F$ is regular or subregular if the exact value of $\al$ in the above definitions is not important.
%If any of the properties in the above definitions is satisfied with an $\al>0$, it is also satisfied with any $\al'\in]0,\al[$.
The exact upper bound of all $\al>0$ such that a property in the above definitions is satisfied with some $\delta\in]0,+\infty]$ and $\mu\in]0,+\infty]$ (and $\eta\in]0,+\infty]$), is called the \emph{modulus} (or rate) of the property.

Apart from the main parameter $\al$, providing a quantitative measure of the respective property,
the properties in above definitions depend also on the auxiliary parameters $\de$, $\eta$ and $\mu$.
They control (directly and indirectly) the size of the \nbh s of $\bx$ and $\bp$ involved in the definitions.
As discussed above, the last parameter can be dropped (together with the corresponding constraints).
We keep all the parameters to emphasize their different roles in the definitions and corresponding characterizations.
The necessary and sufficient regularity conditions presented in the paper normally involve the same collection of parameters.

The properties in \cref{D1.2,D1.3} can be interpreted as kinds of Lipschitz-like properties of the implicit multifunction (solution mapping) \eqref{G}.
This observation opens a way for numerous applications of the characterizations established in this and many other papers; cf. \cref{S6}.

\if{
\AK{26/12/19.
In part (ii) there should be two different $\de$'s or, more generally, two different sets instead of \nbh s.
It seems $P$ does not have to be even a metric space.}

\NDC{27/12/19
Yes, it can be two different $\de$.
Should we do it?}

\AK{26/12/19.
I think, with the above definition, \cref{P2,P1.1} are OK.
What about the rest of the paper?
Is subregularity needed?
$B_\de(\bx)$ can also be replaced by an arbitrary set.
Ioffe (and possibly others) used to write about regularity relative to a set.
There can be issues with \cref{P1}, where $\de+\mu$ is used.}
}\fi
\if{
\NDC{27/12/19
Subregularity does not needed.
I formulated some corollaries for the property just for completeness.
I have changed the rest of the paper according to the property in \cref{D1.2}.
Everything seems OK.
I think there was no issue with \cref{P1} since $\de+\mu$ was related to $x$ and $y$.}
}\fi
\if{
\NDC{6.12.19
I do not know how to define Robinson semiregularity of implicit multifunctions.}
}\fi

Regularity properties of implicit multifunctions were first considered by Robinson \cite{Rob75.2,Rob76,Rob76.2} when studying stability of solution sets of generalized equations.
This initiated a great deal of research by many authors, mostly in normed spaces (and with $\by:=0$).
Dontchev et al. \cite[Theorem~2.1] {DonQuiZla06} gave a
sufficient condition for regularity of implicit multifunctions in terms of graphical derivatives.
Ngai et al. \cite{NgaThe04,NgaTroThe13} employed the theory of error bounds to characterizing the property in
metric and Banach spaces.
In \cite{LedZhu99,LeeTamYen08,YenYao09,HuyYao09,ChuKruYao11, HuyKimNin12, Ngh14, ChuKim16,GfrOut16.2}
dual sufficient conditions were established
in finite and infinite dimensions in terms of Fr\'echet, limiting, directional limiting and Clarke coderivatives.
Chieu et al. \cite{ChiYaoYen10}
established connections between regularity and  Lipschitz-like properties of implicit multifunctions.

The regularity properties of the type given in \cref{D1.2,D1.3} are often referred to in the literature as
\textit{metric regularity}
\cite{ChuKim16,HuyYao09,LeeTamYen08},
\emph{metric regularity in Robinson's sense} \cite{YenYao09,Ngh14}, and \textit{Robinson metric regularity}
\cite{ChiYaoYen10,HuyKimNin12} (of implicit multifunctions).
Following Ioffe \cite{Iof17,Iof17.1}, we prefer to talk about
\textit{uniform regularity}.
We refer the readers to \cite{LedZhu99,AzeCorLuc02,AzeBen08,Ngh14,Iof17.1,DonRoc14} for more discussions and historical comments.

The metric properties in \cref{D1.3} admit equivalent geometric characterizations.
This is illustrated by the next proposition providing a characterization for the property in \cref{D1.2}(ii).

\begin{proposition}%\label{P1.1}
Let $X$ and $Y$ be metric spaces, and $P$ be a set, $F:P\times X\rightrightarrows Y$, $\bx\in X$, $\by\in Y$, and $\al>0$.
The mapping $F$ is
$\al-$subregular in $x$ uniformly in $p$ over $P$ at $(\bx,\by)$ with some $\delta\in]0,+\infty]$ and $\mu\in]0,+\infty]$ if and only if
\begin{align}\label{P2-1}
F_p\iv(\by)\cap B_\rho(x)\ne \es
\end{align}
for all $\rho\in]0,\mu[$, $p\in P$ and $x\in B_\de(\bx)$ with $d(\by,F(p,x))<\al\rho$.
\end{proposition}

\begin{proof}
Suppose $F$ is
$\al-$subregular in $x$ uniformly in $p$ over $P$ at $(\bx,\by)$ with some ${\delta\in]0,+\infty]}$ and $\mu\in]0,+\infty]$.
Let ${\rho\in]0,\mu[}$, $p\in P$ and $x\in B_\de(\bx)$ with $d(\by,F(p,x))<\al\rho$.
Then $d(\by,F(p,x))<\al\mu$.
By \cref{D1.2}(ii), $d(x,F_p\iv(\by))\le\al\iv d(\by,F(p,x))<\rho$.
Hence, condition \eqref{P2-1} is satisfied.
\sloppy

Conversely, suppose $\delta\in]0,+\infty]$ and $\mu\in]0,+\infty]$, and condition \eqref{P2-1} is satisfied for all  ${\rho\in]0,\mu[}$, $p\in P$ and $x\in B_\de(\bx)$ with $d(\by,F(p,x))<\al\rho$.
Let $p\in P$ and $x\in B_\de(\bx)$ with $d(\by,F(p,x))<\al\mu$.
Choose a $\rho$ satisfying $\al\iv d(\by,F(p,x))<\rho<\mu$.
Then, by \eqref{P2-1}, $d(x,F_p\iv(\by))<\rho$.
Letting $\rho\downarrow \al\iv d(\by,F(p,x))$, we arrive at \eqref{D1.2-2}, i.e. $F$ is
$\al-$subregular in $x$ uniformly in $p$ over $P$ at $(\bx,\by)$ with $\de$ and $\mu$.
\sloppy
\end{proof}

The aim of this paper is not to add some new sufficient or necessary conditions for regularity properties of general \SVM s or implicit multifunctions to the large volume of existing ones (although some conditions in the subsequent sections are indeed new), but to propose a
%comprehensive
unifying general
(i.e. not assuming the mapping $F$ to have any particular structure and not using tangential approximations of $\gph F$) view on the theory of regularity, and clarify the relationships between the existing conditions including their hierarchy.
We expose the typical sequence of regularity assertions, often hidden in the proofs, and
the roles of the assumptions involved in the assertions, in particular, on the underlying space: general metric, normed, Banach or Asplund.

We present a series of necessary and sufficient regularity conditions with the main emphasis (in line with the current trend in the literature) on the latter ones.
The (typical) sequence of sufficient regularity conditions is represented by the following chain of assertions, each subsequent assertion being a consequence of the previous one:
\begin{enumerate}
\item
nonlocal primal space
%slope type
conditions in complete metric spaces (\cref{P1}(ii));
\item
local primal space
%slope type
conditions in complete metric spaces (\cref{C2.2}(ii));
\item
subdifferential conditions in Banach and Asplund spaces (\cref{P5});
\item
normal cone conditions in Banach and Asplund spaces (\cref{T2});
\item
coderivative conditions in Banach and Asplund spaces (\cref{C3.3,C3.4}).
\end{enumerate}
Even if one targets coderivative conditions, they still have to go through the five steps listed above with details often hidden in long proofs.
Apart from making the whole process more transparent, which is our main objective, the assertions in (i)--(iv) can be of independent interest, at least theoretically, especially
the slope type conditions in (ii) and normal cone conditions in (iv).
In combination with tangential approximations of $\gph F$, they are likely to lead to verifiable regularity conditions.

The implications (i) \folgt (ii) and (iv) \folgt (v) in the above list follow immediately from the definitions.
The main assertions are the sufficiency of condition (i), and implications (ii) \folgt (iii) \folgt (iv).
They employ the following fundamental tools of variational analysis:
\begin{itemize}
\item
\emph{\EVP} (sufficiency of condition (i));
\item
\emph{sum rules} for respective subdifferentials (implications (ii) \folgt (iii) \folgt (iv)).
\end{itemize}
Thus, all the sufficient conditions on the list are consequences of the \EVP, and as such, they are `outer' conditions, i.e. they need to be checked at points outside the solution set $F_p\iv(\by)$.

Most of the sufficient conditions are accompanied by the corresponding necessary ones.
The necessary conditions do not require the underlying spaces to be complete and are generally easy consequences of the definitions.
With the exception of the general nonlocal condition in \cref{P1}(i), such conditions are formulated in normed spaces and assume the graph of $F$ to be convex.
In \cref{S4.2}, we provide a series of dual necessary regularity conditions for \SVM s with closed convex graphs acting between Banach spaces some of which are also sufficient.

In the setting of complete metric spaces, and assuming that $\gph F_p$ is closed for all $p\in P$,
the gap between the nonlocal necessary and sufficient subregularity conditions in \cref{P1} is not big:
they share the same inequality \eqref{P1-1}; with all the other parameters coinciding, the sufficiency part naturally requires it to hold for all $x$ in a larger set.
Unfortunately, unlike the `full' regularity possessing the well known coderivative criterion (see, e.g., \cite{Kru88,Mor06.1}), this is not the case in general with local subregularity conditions unless the graph of $F$ is convex.
The sufficient subregularity conditions presented in the paper are the weakest possible in each group, but can still be far from necessary.
As it has been discussed in the literature (see, e.g., a discussion of the equivalent subtransversality property in \cite{KruLukTha17}), the reason for this phenomenon lies in the fact that the subregularity property lacks robustness.

The hot topic of regularity of a \SVM\ $F$ with a special structure, particularly in the arising in numerous applications such as, e.g., KKT systems and variational inequalities, case when $F=g+G$ with $g$ single valued and $G$ set-valued (typically a normal cone mapping), is outside the scope of the current paper.
Computing `slopes' and coderivatives of such mappings (or normal cones to their graphs) is usually a difficult job and requires imposing additional assumptions on $g$ and $G$.
This is what people working in this area normally do.
We want to emphasize that this type of conditions still fall into the five-point scheme described above.

The rest of the paper is organized as follows.
The next \cref{sect1.2} contains some preliminary facts used throughout the paper.
\cref{S4,S5} are dedicated, respectively, to primal and dual sufficient and necessary conditions for the regularity properties.
In \cref{S6},
we illustrate the theory by characterizing
the conventional metric regularity and subregularity of set-valued mappings as well
as stability properties of solution mappings to parametric inclusions.

%%%%%%%%%%%%%%%%%%%%%%%%%%%%%%%%%%%%%%%%%%%%%%%

\section{Preliminaries}\label{sect1.2}

Our basic notation is standard, see, e.g., \cite{Mor06.1,RocWet98,DonRoc14}.
Throughout the paper, if not explicitly stated otherwise, $P$ is an arbitrary set, $X$ and $Y$ are either metric or normed/Ba\-nach/Asplund spaces.
Products of metric or normed spaces are assumed to be equipped with the maximum distance or norm.
The topological dual of a normed space $X$ is denoted by $X^*$, while $\langle\cdot,\cdot\rangle$ denotes the bilinear form defining the pairing between the two spaces.
In a primal space, the open and closed balls with center $x$ and radius $\delta\in]0,+\infty]$ are denoted, respectively, by $B_\delta(x)$ and $\overline{B}_\de(x)$, while $\B$ and $\overline{\B}$ stand for, respectively, the open and closed unit balls.
The open unit ball in the dual space is denoted by $\B^*$.
%If not explicitly stated otherwise, products of metric or normed spaces are assumed to be equipped with the maximum distance or norm.
%$\|(x,y)\|:=\max\{\|x\|,\|y\|\}$, $(x,y)\in X\times Y$.
Symbols $\R$, $\R_+$ and $\N$ stand for the real line, the set of all nonnegative reals, and the set of all positive integers, respectively.

For a set $\Omega$ in a normed space, its closure is denoted by $\cl\Omega$.
The distance from a point $x$ to $\Omega$ is defined by $d(x,\Omega):=\inf_{u\in\Omega}\|u-x\|$, and we use the convention $d(x,\emptyset)=+\infty$.
The indicator function of $\Omega$ is defined by $i_\Omega(x)=0$ if $x\in\Omega$, and $i_\Omega(x)=+\infty$ if $x\notin\Omega$.

The dual conditions in the paper are formulated in terms of \Fr\ and Clarke normals and subdifferentials; cf., e.g., \cite{Kru03,Cla83}.

Given a point $\bx\in \Omega$, the sets
\begin{gather}\label{NC}
N_{\Omega}^F(\bx):= \left\{x^\ast\in X^\ast\mid
\limsup_{\Omega\ni x\to\bar x,\,x\ne \bx} \frac {\langle x^\ast,x-\bx\rangle}
{\|x-\bx\|} \le 0 \right\},
\\\label{NCC}
N_{\Omega}^C(\bx):= \left\{x^\ast\in X^\ast\mid
\ang{x^\ast,z}\le0
\qdtx{for all}
z\in T_{\Omega}^C(\bx)\right\}
\end{gather}
are, respectively, the \emph{Fr\'echet} and \emph{Clarke normal cones} to $\Omega$ at $\bx$.
In definition \eqref{NCC}, $T_{\Omega}^C(\bx)$ stands for the \emph{Clarke tangent cone} to $\Omega$ at $\bx$.
The sets \eqref{NC} and \eqref{NCC} are nonempty
closed convex cones satisfying
$N_{\Omega}^F(\bx)\subset N_{\Omega}^C(\bx)$.
If $\Omega$ is a convex set, both cones reduce to the normal cone in the sense of convex analysis:
\begin{gather*}\label{CNC}
N_{\Omega}(\bx):= \left\{x^*\in X^*\mid \langle x^*,x-\bx \rangle \leq 0 \qdtx{for all} x\in \Omega\right\}.
\end{gather*}

For an extended-real-valued function $f:X\to\R\cup\{+\infty\}$ on a normed space, its domain and epigraph are defined,
respectively, by $\dom f:=\{x\in X\mid f(x)< +\infty\}$ and $\epi f:=\{(x,\alpha)\in X\times \mathbb{R}\mid f(x)\le\alpha\}$.
The \emph{Fr\'echet} and \emph{Clarke subdifferentials} of $f$ at $\bar x\in\dom f$ are defined, respectively, as
\begin{gather}\label{sdF}
\partial^F f(\bar x):=\left\{x^*\in X^*\mid \liminf_{\substack{x\to \bar x,\,x\ne\bx}} \dfrac{f(x)-f(\bar x)-\langle x^*,x-\bar x\rangle}{\|x-\bar x\|}\ge 0\right\},
\\\label{sdC}
\partial^Cf(\bx):=\left\{x^*\in X^*\mid \langle x^*,z\rangle \le f^\circ(\bx,z)
\quad\text{for all}\quad
z\in X\right\},
\end{gather}
where $f^\circ(\bx,z)$ is the \emph{Clarke--Rockafellar directional derivative} \cite{Roc79,Roc80}
of $f$ at $\bx$ in the direction $z\in X$.
The sets \eqref{sdF} and \eqref{sdC} are closed and convex, and satisfy
$\partial^F{f}(\bx)\subset\partial^C{f}(\bx)$.
If $f$ is convex, they
reduce to the subdifferential in the sense of convex analysis:
\begin{gather*}
\partial{f}(\bx):= \left\{x^\ast\in X^\ast\mid
f(x)-f(\bx)-\langle{x}^\ast,x-\bx\rangle\ge 0 \qdtx{for all} x\in X \right\}.%\label{Frsd}
\end{gather*}
It is easy to check that $N_{\Omega}^F(\bx)=\partial^Fi_\Omega(\bx)$, $N_{\Omega}^C(\bx)=\partial^Ci_\Omega(\bx)$, and
\begin{gather*}
\sd^F f(\bx)=\left\{x^*\in X^*\mid (x^*,-1)\in N^F_{\epi f}(\bx,f(\bx))\right\},\\
\partial^C{f}(\bx)= \left\{x^\ast\in X^\ast\mid
(x^*,-1)\in N_{\epi f}^C(\bx,f(\bx))\right\}.
\end{gather*}
By convention, we set $N_{\Omega}^F(\bx) =N_{\Omega}^C(\bx):=\es$ if $\bx\notin \Omega$ and $\partial^F{f}(\bx)=\partial^C{f}(\bx):=\es$ if $x\notin\dom f$.
We often use the generic notations $N$ and $\sd$ for  Fr\'echet and Clarke objects, specifying wherever necessary the type of the object by an appropriate superscript, e.g., $N:=N^F$ or $N:=N^C$.

The following fact is an immediate consequence of the definition of the \Fr\ subdifferential; cf. \cite{Kru03,Mor06.1}.

\begin{lemma}\label{L2.3}
Suppose $X$ is a normed space and $f:X\to\R\cup\{+\infty\}$.
If $\bx\in\dom f$ is a point of local minimum of $f$, then $0\in\sd^Ff(\bx)$.
\end{lemma}

The representation of the (convex) subdifferential of a norm in the next lemma is of importance; cf. \cite[Corollary~2.4.16]{Zal02}.

\begin{lemma} \label{L3}
Let $(Y,\|\cdot\|)$ be a normed space.
Then
\begin{enumerate}
\item
$\sd\|\cdot\|(0)=\{y^*\in Y^*\mid
\|y^*\|\le 1\}$;
\item
$\sd\|\cdot\|(y)=\{y^*\in Y^*\mid \langle y^*,y\rangle=\|y\|\;\;
\text{and}
\;\;
\|y^*\|= 1\},
\;\;
y\ne 0$.
\end{enumerate}
\end{lemma}

For an extended real-valued function $f$ on a metric space,
its \emph{slope} and \emph{nonlocal slope} (cf. \cite{NgaThe08,Iof00,AzeCorLuc02,Kru15}) at $x\in\dom f$ are defined, respectively, by
\begin{align*}
|\nabla f|(x):=\limsup_{u\rightarrow x,u\ne x}\dfrac{ [f(x)-f(u)]_+}{d(x,u)} \AND
|\nabla f|^\diamond(x):=\sup\limits_{u\ne x}
\dfrac{[f(x)-f_+(u)]_+}{d(x,u)},
\end{align*}
where $\al_+:=\max\{0,\al\}$ for any $\al\in\R$.
If $x\notin\dom f$, we set ${|\nabla f|(x)=|\nabla f|^\diamond(x):=+\infty}$.
The following simple facts are well known; cf., e.g., \cite{CuoKru21}.

\begin{lemma}\label{P1.3}
Let $X$ be a metric space, $f:X\to\R\cup\{+\infty\}$, $x\in\dom f$,
and ${f(x)>0}$.
\begin{enumerate}
\item
$|\nabla f|(x)\le |\nabla f|^\diamond(x)$.
\item
If $X$ is a normed space and $f$ is convex,
then
$|\nabla f|^\diamond(x)=|\nabla f|(x)=d(0,\partial f(x)).$
\end{enumerate}
\end{lemma}

A set-valued mapping $F:X\rightrightarrows Y$ between two sets $X$ and $Y$ is a mapping, which assigns to every $x\in X$ a subset (possibly empty) $F(x)$ of $Y$.
We use the notations $\gph F:=\{(x,y)\in X\times Y\mid
y\in F(x)\}$ and $\dom\:F:=\{x\in X\mid F(x)\ne\emptyset\}$
for the graph and the domain of $F$, respectively, and $F^{-1}:Y\rightrightarrows X$ for the inverse of $F$.
This inverse (which always exists with possibly empty values at some $y$) is defined by $F^{-1}(y):=\{x\in X \mid y\in F(x)\}$, $y\in Y$. Obviously $\dom F^{-1}=F(X)$.

If $X$ and $Y$ are normed spaces, the \emph{coderivative} of $F$ at $(x,y)\in\gph F$ is a set-valued mapping $D^*F(x,y):Y^*\rightrightarrows X^*$ defined by
\begin{align}\label{coder}
D^*F(x,y)(y^*):=\{x^*\in X^*\mid (x^*,-y^*)\in N_{\gph F}(x,y)\},
\quad
y^*\in Y^*.
\end{align}
Depending on the type of the normal cone in \eqref{coder}, it can define various coderivatives.
We use symbols $D_F^*$ and $D_C^*$ to denote, respectively, the Fr\'echet and Clarke coderivatives.
\if{
Let $\la\ge 0$, and $\varepsilon\ge 0$.
Define the approximate normalized $\lambda-$enlargement of the duality mapping (cf. \cite{ChuKim16})
\todo{$X$ or $Y$?}
$J_\lambda(\cdot,\varepsilon):
X\rightrightarrows X^*$ for any $x\ne 0$
by
\begin{align*}
J_\lambda(x,\varepsilon):=
\Big\{\dfrac{u^*+\lambda v^*}{\|u^*+\lambda v^*\|}\mid z\in X\setminus\{0\},\,v^*\in X^*,\,\|z-x\|\le\varepsilon,\,
u^*\in J(z),\,\|v^*\|\le 1 \Big\}.
\end{align*}
When $\varepsilon=0$, the mapping is denoted by $J_\la(\cdot)$ and called the normalized $\lambda-$enlargement of the duality mapping (cf. \cite{LiMor12}).\sloppy
}\fi

The key tools in the proofs of the main results are
the celebrated Ekeland variational principle and several subdifferential sum rules; cf. \cite{DonRoc14,Mor06.1,Iof17,IofTik79,Zal02,Roc79,Kru03,Fab89}.

\begin{lemma}\label{EVP}
%[Ekeland variational principle]
Suppose $X$ is a complete metric space, $f: X\to \mathbb{R} \cup \{ +\infty\}$ is lower semicontinuous,
$x\in X$, $\varepsilon>0$
and $\lambda>0$. If
$f(x)<\inf_{X} f+\varepsilon$,
then there exists an $\hat x\in X$ such that
\begin{enumerate}
\item
$d(\hat{x},x)<\lambda$;
\item
$f(\hat{x})\le f(x)$;
\item
$f(u)+(\varepsilon/\lambda)d(u,\hat{x})\ge f(\hat{x})$ for all $u\in X.$
\end{enumerate}
\end{lemma}

\begin{lemma}
%[Subdifferential sum rules]
\label{SR}
Suppose $X$ is a normed space, $f_1,f_2:X\to\R \cup\{+\infty\}$, and $\bx\in\dom f_1\cap\dom f_2$.
\begin{enumerate}
\item
%{\bf Convex sum rule}.
Let $f_1$ and $f_2$ be convex and $f_1$ be continuous at a point in $\dom f_2$.
Then
$$\partial(f_1+f_2)(\bx)=\sd f_1(\bx)+\partial f_2(\bx).$$
\item
%{\bf Clarke--Rockafellar sum rule}.
Let
$f_1$ be Lipschitz continuous and $f_2$ be lower semicontinuous in a neighbourhood of $\bx$.
Then
$$\partial^C(f_1+f_2)(\bx)\subset\sd^C f_1(\bx) +\partial^Cf_2(\bx).$$
\item
%{\bf Fuzzy sum rule}.
Let $X$ be Asplund,
$f_1$ be Lipschitz continuous and $f_2$ be lower semicontinuous in a neighbourhood of $\bx$.
Then, for any $x^*\in\partial^F(f_1+f_2)(\bx)$ and $\varepsilon>0$, there exist $x_1,x_2\in X$ with $\|x_i-\bx\|<\varepsilon$, $|f_i(x_i)-f_i(\bx)|<\varepsilon$ $(i=1,2)$, such that
$$x^*\in\partial^Ff_1(x_1) +\partial^Ff_2(x_2)+\varepsilon\B^\ast.$$
\end{enumerate}
\end{lemma}

Recall that a Banach space is \emph{Asplund} if every continuous convex function on an open convex set is Fr\'echet differentiable on a dense subset \cite{Phe93}, or equivalently, if the dual of each its separable subspace is separable.
We refer the reader to \cite{Phe93,Mor06.1,BorZhu05} for discussions about and characterizations of Asplund spaces.
All reflexive, particularly, all finite dimensional Banach spaces are Asplund.

\if{
The next lemma is needed for our subsequent analysis.
\begin{lemma}\label{L1.6}
Let $\Omega$ be a subset of a metric space $X$, $f:X\rightarrow\R\cup\{+\infty\}$ be continuous.
Then
$\sup_{\Omega}f
=\sup_{\cl\Omega}f$.
\end{lemma}
\begin{proof}
It is straightforward that
$\sup_{x\in\Omega}f(x)
\le\sup_{x\in\cl\Omega}f(x)$.
Let $x\in\cl\Omega$.
Then there exists a sequence $\{x_k\}_{k\in\N}$ in $\Omega$ converges to $x$.
In view of the continuity of $f$, we have
$f(x)=\lim_{k\to+\infty}f(x_k)
\le \sup_{\Omega}f.$
Hence, $\sup_{\cl\Omega}f
\le\sup_{\Omega}f$.
The proof is complete.
\end{proof}
}\fi

\section{Slope Necessary and Sufficient Conditions}\label{S4}

This section is dedicated to slope necessary and sufficient conditions.
For simplicity, we focus on the uniform subregularity property in \cref{D1.2}(ii).
The corresponding conditions for the property in \cref{D1.2}(i) can be formulated in a similar way.
Besides, in view of \cref{R1.1}, in normed spaces (which is our setting in the next section) such conditions can be obtained as consequences of those for the subregularity.

The necessary conditions are deduced directly from the definitions of the respective properties, while the sufficient ones come from the application of the \EVP.
In the convex case, the conditions are necessary  and sufficient.

In this section, $P$ is a nonempty set, $X$ and $Y$ are metric spaces, and
$F:P\times X\rightrightarrows Y$.
We assume the parameters $\bx\in X$, $\by\in Y$,
$\al>0$, ${\delta\in]0,+\infty]}$ and $\mu\in]0,+\infty]$ to be fixed.
In what follows, we employ a collection of functions
\begin{gather}\label{psi}
\psi_p(u,v):={d(v,\by)}+i_{\gph F_{p}}(u,v),
\quad
u\in X,\;v\in Y
\end{gather}
depending on a parameter $p\in P$.
Along with the standard maximum {distance} on $X\times Y$, we also use a {metric} depending on a parameter $\ga>0$:
\begin{gather}\label{pdist}
d_\ga((u,v),(x,y)) :=\max\left\{d(u,x),\ga d(v,y)\right\},\quad
u,x\in X,\;v,y\in Y.
\end{gather}

The next theorem plays a crucial role for the subsequent considerations.
The slope and subdifferential/normal cone/coderivative
conditions for uniform $\al-$subregularity in this paper are consequences of this theorem.

\begin{theorem}\label{P1}
\begin{enumerate}
\item
If $F$ is $\al-$subregular in $x$ uniformly in $p$ over $P$ at $(\bx,\by)$ with ${\delta}$ and $\mu$, then
\begin{align}\label{P1-1}
\sup_{\substack{(u,v)\in\gph F_p,\,(u,v)\ne (x,y)\\
d(u,\bx)<\de+\mu,\,d(v,\by)<\al\mu}}
{\dfrac{d(y,\by)-d(v,\by)}{d_\ga((u,v),(x,y))}}\ge\al
\end{align}
for $\gamma:=\al\iv$, and all $p\in P$, $x\in B_{\de}(\bx)$ and $y\in Y$ satisfying
\begin{align}\label{P4.1-1}
x\notin F_p\iv(\by),\quad
y\in F(p,x){\cap B_{\al\mu}(\by)}.
\end{align}

\item
Suppose $X$ and $Y$ are complete, and $\gph F_p$ is closed for all $p\in P$.
If inequality \eqref{P1-1} holds for some
$\ga>0$, and all $p\in P$, $x\in B_{\de+\mu}(\bx)$ and $y\in Y$ satisfying \eqref{P4.1-1}, then
$F$ is $\al-$subregular in $x$ uniformly in $p$ over $P$ at $(\bx,\by)$ with $\delta$ and $\mu$.
\end{enumerate}
\end{theorem}

\begin{proof}
\begin{enumerate}
\item
Suppose $F$ is $\al-$subregular in $x$ uniformly in $p$ over $P$ at $(\bx,\by)$ with $\delta$ and $\mu$.
Let $p\in P$, $x\in B_{\de}(\bx)$ and $y\in Y$ satisfy \eqref{P4.1-1}, ${\ga:=\al\iv}$, and $\eta>1$.
By \eqref{P4.1-1} and \cref{D1.2}(ii), there exist
a $\xi\in]1,\eta[$ such that $\xi{d(y,\by)}<\al\mu$,
and a point $\hat{x}\in F_p\iv(\by)$ such that $\al{d(x,\hat{x})}<\xi{d(y,\by)}$.
Thus, $(\hat{x},\by)\in\gph F_p$, $(\hat{x},\by)\ne(x,y)$,
\begin{gather*}
d(\hat{x},\bx)\le d(\hat{x},x)+d(x,\bar{x})<\al\iv\xi
d(y,\by)+\delta
\le\mu+\de,\quad\text{and}\\
d_\ga((x,y),(\hat{x},\by))=\max\{
d(x,\hat{x}),\ga d(y,\by)\}\le\al\iv\max\{\xi,1\}d(y,\by)=\al\iv\xi d(y,\by).
\end{gather*}
Hence,
\begin{gather*}
\sup_{\substack{(u,v)\in\gph F_p,\,(u,v)\ne (x,y)\\
d(u,\bx)<\de+\mu,\,d(v,\by)<\al\mu}}
\dfrac{d(y,\by)-d(v,\by)}{d_\ga((u,v),(x,y))}\ge
\dfrac{d(y,\by)}{d_\ga((\hat{x},\by),(x,y))}
\ge\al\xi\iv>\al\eta\iv.
\end{gather*}
Letting $\eta\downarrow 1$, we arrive at \eqref{P1-1}.

\item
Suppose $F$ is not $\al-$subregular in $x$ uniformly in $p$ over $P$ at $(\bx,\by)$ with $\delta$ and $\mu$.
By \cref{D1.2}(ii), there exist points $p\in P$ and $x\in B_\de(\bx)$ such that
\begin{align*}
d(\by,F(p,x))<\al\min\{d(x,F_p\iv(\by)),\mu\}.
\end{align*}
Hence, $x\notin F_p\iv(\by)$, or equivalently, $\by\notin F(p,x)$.
Set $\mu_0:=\min\{d(x,F_p\iv(\by)),\mu\}$.
Choose a number $\varepsilon$ such that
$d(\by,F(p,x))<\varepsilon<\al\mu_0$, and a point
$y\in F(p,x)$ such that ${d(y,\by)}<\varepsilon$.
The function $\psi_p:X\times Y\rightarrow\R_+\cup\{+\infty\}$, defined by \eqref{psi}, is lower semicontinuous on
$\overline{B}_{\de+\mu}(\bx)\times{\overline{B}_{\al\mu}(\by)}$.
Besides,
\sloppy
\begin{align*}
\psi_p(x,y)={d(y,\by)}< \inf_{\overline{B}_{\de+\mu}(\bx)\times {\overline{B}_{\al\mu}(\by)}}\psi_p+\varepsilon.
\end{align*}
Let $\ga>0$.
Applying the Ekeland variational principle (\cref{EVP}) to the restriction of $\psi_p$ to the complete metric space $\overline{B}_{\de+\mu}(\bx)\times {\overline{B}_{\al\mu}(\by)}$ with the metric \eqref{pdist}, we can find a point $(\hat{x},\hat{y})\in \overline{B}_{\de+\mu}(\bx)\times {\overline{B}_{\al\mu}(\by)}$ such that
\begin{gather}\label{P1-3}
d_\ga((\hat{x},\hat{y}),(x,y))< \mu_0,\quad
\psi_p(\hat{x},\hat{y})\le\psi_p(x,y),
\\\label{P1-5}
\psi_p(\hat{x},\hat{y})\le\psi_p(u,v) +(\varepsilon/\mu_0) d_\ga((u,v),(\hat{x},\hat{y}))
\end{gather}
for all $(u,v)\in \overline{B}_{\de+\mu}(\bx)\times {\overline{B}_{\al\mu}(\by)}$.
By \eqref{P1-3},
we  have $\hat{y}\in F(p,\hat{x})$, and
\begin{gather*}
d(\hat{x},\bx)\le d(\hat{x},x)+d(x,\bx)<\mu_0+\de \le\mu+\de,\\
d(\hat{y},\by)\le d(y,\by)<\eps<\al\mu_0\le\al\mu.
\end{gather*}
Besides, $d(\hat{x},x)<\mu_0\le d(x,F_p\iv(\by))$.
This implies $\hat{x}\notin F_p\iv(\by)$, and consequently, ${\hat{y}\ne \by}$.
%Thus, $p$, $\hat x$ and $\hat y$ \red{satisfy \eqref{P4.1-2}}.
It follows from \eqref{P1-5} that
\begin{align*}
\sup_{\substack{(u,v)\in\gph F_p,\, (u,v)\ne(\hat{x},\hat{y})\\
d(u,\bx)<\de+\mu,\,d(v,\by)<\al\mu}}
\dfrac{d(\hat y,\by)-d(v,\by)}{d_\ga((u,v),(\hat{x},\hat{y}))} \le\dfrac{\eps}{\mu_0}<\al.
\end{align*}
%In view of \cref{L1.6},
The last estimate contradicts \eqref{P1-1}.
\end{enumerate}
\end{proof}

\if{
Combining the two parts of \cref{P1}, we arrive at a complete primal space characterization of $\al-$subregularity.
}\fi
%\red{The characterization is new.}
\if{
\begin{corollary}%\label{C2.2}
Let $X$ and $Y$ be complete, and $\gph F_p$ be closed for all $p\in P$.
The mapping $F$ is $\al-$subregular in $x$ uniformly in $p$ over $P$ at $(\bx,\by)$ with some $\delta\in]0,+\infty]$ and $\mu\in]0,+\infty]$ if and only if inequality \eqref{P1-1} holds
with $\ga:=\al\iv$ for all $p$, $x$ and $y$ satisfying~\eqref{P4.1-1}.
\end{corollary}
}\fi
\begin{remark}\label{R2.1}
\begin{enumerate}
\item
The expression in the left-hand side of the inequality \eqref{P1-1} is the nonlocal $\ga$-slope \cite[p.~60]{Kru15} at $(x,y)$ of the restriction of the function $\psi_p$, given by \eqref{psi}, to $\gph F_p\cap[B_{\de+\mu}(\bx)\times B_{\al\mu}(\by)]$.
\item
By the definition of the {metric} \eqref{pdist}, if inequality \eqref{P1-1} is satisfied with a $\ga>0$, then it is also satisfied with any $\ga'\in]0,\gamma[$.
This observation is applicable to all slope inequalities in this section.
%\item
%The sufficient condition part (i) of
%\cref{P1} is also necessary for the property.

\item
The completeness of the space and closedness assumption in part (ii) of \cref{P1} (and the subsequent statements) can be relaxed: it suffices to require that $\gph F_p\cap [\overline{B}_{\de+\mu}(\bx)\times \overline{B}_{\al\mu}(\by)]$ is complete
%\red{and closed}
for all $p\in P$.

\item
The sufficient condition in part (ii) of \cref{P1} is often hidden in the proofs of dual sufficient conditions.
%while the necessary condition in part (i) is new.

\item
When $X$ and $Y$ are complete, and $\gph F_p$ is closed for all $p\in P$,
the gap between the nonlocal necessary and sufficient regularity conditions in parts (i) and (ii) of \cref{P1} is not big:
they share the same inequality \eqref{P1-1}; with all the other parameters coinciding, the necessity part (i) guarantees this inequality to hold for all $x\in B_{\de}(\bx)$, while the sufficiency part (ii) requires it to hold for all $x$ in a larger set $B_{\de+\mu}(\bx)$.
\end{enumerate}
\end{remark}

We now illustrate \cref{P1} by applying it to the local (in $p$) setting in \cref{D1.3}(ii).
The application is straightforward.
We provide a single illustration of this kind, although the other statements in this and the next section are also applicable to this setting.

\begin{corollary}%\label{C2.2}
Let $P$ be a metric space, $(\bp,\bx,\by)\in\gph F$ and $\eta\in]0,+\infty]$.
\begin{enumerate}
\item
If $F$ is $\al-$subregular in $x$ uniformly in $p$ at $(\bp,\bx,\by)$ with $\eta$, $\delta$ and $\mu$, then inequality \eqref{P1-1} holds with $\gamma:=\al\iv$
for all ${p\in B_\eta(\bp)}$, $x\in B_{\de}(\bx)$ and $y\in Y$ satisfying \eqref{P4.1-1}.

\item
Suppose $X$ and $Y$ are complete, and $\gph F_p$ is closed for all $p\in B_\eta(\bp)$.
If inequality \eqref{P1-1} holds for some $\ga>0$, and all $p\in B_\eta(\bp)$, $x\in B_{\de+\mu}(\bx)$ and $y\in Y$ satisfying \eqref{P4.1-1}, then
$F$ is $\al-$subregular in $x$ uniformly in $p$ at $(\bp,\bx,\by)$ with $\eta$, $\delta$ and $\mu$.
\end{enumerate}
\end{corollary}

The next statement presents a localized version of \cref{P1}.

\begin{corollary}\label{C2.2}
\begin{enumerate}
\item
Suppose $X$ and $Y$ are normed spaces, and $\gph F_p$ is convex for all $p\in P$.
If $F$ is $\al-$subregular in $x$ uniformly in $p$ over $P$ at $(\bx,\by)$ with ${\delta}$ and $\mu$, then
\begin{align}\label{C1-1}
\limsup_{\substack{u\to x,\,v\to y,\,(u,v)\in\gph F_p,\,(u,v)\ne (x,y)\\
d(u,\bx)<\de+\mu,\,d(v,\by)<\al\mu}}
{\dfrac{d(y,\by)-d(v,\by)}{d_\ga((u,v),(x,y))}}\ge\al
\end{align}
for $\gamma:=\al\iv$, and all
$p\in P$, $x\in B_{\de}(\bx)$ and $y\in Y$ satisfying \eqref{P4.1-1}.

\item
Suppose $X$ and $Y$ are complete, and $\gph F_p$ is closed for all $p\in P$.
If inequality \eqref{C1-1} holds for some $\ga>0$, and all $p\in P$, $x\in B_{\de+\mu}(\bx)$ and $y\in Y$ satisfying \eqref{P4.1-1}, then
$F$ is $\al-$subregular in $x$ uniformly in $p$ over $P$ at $(\bx,\by)$ with $\delta$ and $\mu$.
\end{enumerate}
\end{corollary}

\begin{proof}
%\emph{of \cref{C2.2}}.
In view of \cref{R2.1}(i) and \cref{R2.2}(i),
assertion (i) follows from \cref{P1.3}(ii) and \cref{P1}(i), while assertion (ii) is a consequence of \cref{P1.3}(i) and
\cref{P1}(ii).
\end{proof}

\begin{remark}\label{R2.2}
\begin{enumerate}
\item
The expression in the left-hand side of inequality \eqref{C1-1} is the $\ga$-slope \cite[p.~61]{Kru15}  at $(x,y)$ of the restriction of the function $\psi_p$, given by \eqref{psi}, to $\gph F_p\cap[B_{\de+\mu}(\bx)\times B_{\al\mu}(\by)]$.
\item
The convexity assumption in part (i) of \cref{C2.2} (and the subsequent statements) can be relaxed: it suffices to require that $\gph F_p\cap [\overline{B}_{\de+\mu}(\bx)\times \overline{B}_{\al\mu}(\by)]$ is convex for all $p\in P$.

\item
In the particular case when $P$ is a neighborhood of a point $\bar p$ in some metric space, part (ii) of \cref{C2.2} is a quantitative version of \cite[Proposition~3.5]{Iof17.1}.
Ngai et al. \cite[Theorem~3]{NgaTroThe13} established a primal sufficient condition for the property under the assumption that the mapping $F(\cdot,\bx)$ is lower semicontinuous at $\bar p$.
%In our current approach, this assumption is not needed.
\end{enumerate}
\end{remark}

%\begin{corollary}
%Let $X$ and $Y$ be Banach spaces, $\al>0$, $\delta\in]0,+\infty]$ and $\mu\in]0,+\infty]$.
%Suppose $\gph F_p$ is closed and convex for all $p\in %B_\de(\bar p)$.
%The mapping $G$ is Robinson %$\al-$regular at $\bx$ %with $\de$ and $\mu$ if and only %if, for some $\ga>0$, inequality %\eqref{C1-1} holds for all $p$, $x$ and $y$ %satisfying \eqref{P4.1-1}.
%\end{corollary}

\if{
Combining the two parts of \cref{C2.2}, we arrive at a complete primal space infinitesimal characterization of uniform $\al-$subregularity in the convex setting.
}\fi
%\red{The characterization is new.}

\if{
\begin{corollary}\label{T4.1}
Suppose $X$ and $Y$ are Banach spaces, and $\gph F_p$ is closed and convex for all $p\in P$.
The mapping $F$ is $\al-$subregular in $x$ uniformly in $p$ over $P$ at $(\bx,\by)$ if and only if
inequality \eqref{C1-1}
holds with $\ga=\al\iv$ for all $p$, $x$ and $y$ satisfying \eqref{P4.1-1}.
\end{corollary}
}\fi
\section{Dual Necessary and Sufficient Conditions}\label{S5}

In this section, we continue studying the mapping
$F:P\times X\rightrightarrows Y$ where $P$ is a nonempty set, while
$X$ and $Y$ are assumed to be normed spaces.
We also assume the parameters $\bx\in X$, $\by\in Y$,
$\al>0$, ${\delta\in]0,+\infty]}$ and $\mu\in]0,+\infty]$ to be fixed, and the collection of functions $\psi_p$ be defined by~\eqref{psi}.

The primal and dual parametric product space norms, corresponding to the distance \eqref{pdist}, have the following form:
\begin{gather}\label{pnorm}
\|(x,y)\|_{\ga}=\max\{\|x\|,{\ga}\|y\|\},\quad
x\in X,\;y\in Y,
\\\label{dnorm}
\|(x^*,y^*)\|_{\ga}=\|x^*\|+\ga\iv\|y^*\|,\quad
x^*\in X^*,\;y^*\in Y^*.
\end{gather}
We denote by $d_\ga$ the distance in $X^*\times Y^*$ determined by \eqref{dnorm}.

\subsection{Dual Sufficient Conditions}

In this subsection, we assume additionally that $X$ and $Y$ are Banach spaces, and $\gph F_p$ is closed for all $p\in P$.

The next subdifferential sufficient condition for uniform $\al-$subregularity is a consequence of \cref{C2.2}(ii) thanks to the subdifferential sum rules in \cref{SR}.

\begin{proposition}\label{P5}
Let $\sd:=\sd^C$.
If
\begin{align}\label{P5-1}
d_\gamma\left(0,\sd\psi_p(x,y)\right)\ge\al
\end{align}
for some $\ga>0$, and all $p\in P$, $x\in B_{\de+\mu}(\bx)$ and $y\in Y$ satisfying \eqref{P4.1-1}, then $F$ is $\al-$subregular in $x$ uniformly in $p$ over $P$ at $(\bx,\by)$ with $\delta$ and~$\mu$.

If $X$ and $Y$ are Asplund, then the above assertion is valid with $\sd:=\sd^F$.
\end{proposition}

\begin{proof}
Suppose $F$ is not $\al-$subregular in $x$ uniformly in $p$ over $P$ at $(\bx,\by)$ with $\delta$ and $\mu$.
Let $\ga>0$.
By \cref{C2.2}(ii), there exist points $p\in P$, $x\in B_{\de+\mu}(\bx)$ and $y\in Y$ satisfying \eqref{P4.1-1}, and an $\al'\in]0,\al[$ such that
\begin{align*}
\|y-\by\|-\|v-\by\|\le\al'\|(u,v)-(x,y)\|_\ga
\end{align*}
for all $(u,v)\in\gph F_p\cap [B_{\de+\mu}(\bx)\times B_{\al\mu}(\by)]$ near $(x,y)$.
In other words, $(x,y)$ is a local minimizer of the function
\begin{align}\label{P5P1}
(u,v)\mapsto\psi_p(u,v)+\al'\|(u,v)-(x,y)\|_\ga.
\end{align}
By \cref{L2.3}, its \Fr\ and, as a consequence, Clarke subdifferential at this point contains $0$.
Observe that \eqref{P5P1} is the sum of the function $\psi_p$ and the Lipschitz continuous convex function $(u,v)\mapsto\al'\|(u,v)-(x,y)\|_{\ga}$, and, by \cref{L3}, at any point all subgradients $(x^*,y^*)$ of the latter function satisfy $\|(x^*,y^*)\|_{\ga}\le\al'$.
By \cref{SR}(ii), there exists a subgradient $(x^*,y^*) \in \sd^C\psi_p(x,y)$ such that $\|(x^*,y^*)\|_{\ga}\le\al'<\al$, which contradicts~\eqref{P5-1}.

Let $X$ and $Y$ be Asplund.
Choose an $\varepsilon>0$ such that
\begin{align*}
\varepsilon<\min\left\{{\de+\mu}-\|x-\bx\|, {\al\mu-\|y-\by\|},\al-\al',\|y-\by\|/2,d(x,F_p\iv(\by))/2\right\}.
\end{align*}
By \cref{SR}(iii), there exist points $x'\in B_\varepsilon({x}),y'\in B_\varepsilon({y})$ with $(x',y')\in\gph F_p$, and a subgradient $(x^*,y^*)\in \sd^F\psi_p(x',y')$ such that
\begin{align}\label{P3.1-1}
\|(x^*,y^*)\|_{\ga}<\al'+\varepsilon<\al.
\end{align}
Besides,
$x'\in{B_{\delta+\mu}(\bx)}\setminus F_p\iv(\by)$, $\by\ne y'\in B_{\al\mu}(\by)$ as
\begin{gather*}
\|y-\by\|/2<\|y'-\by\|,\quad
d(x,F_p\iv(\by))/2<d(x',F_p\iv(\by)),\\
\|x'-\bx\|\le\|x'-x\|+\|x-\bx\|<{\de+\mu},\quad
\|y'-\by\|\le\|y'-y\|+\|y-\by\|<\al\mu.
\end{gather*}
It follows from \eqref{P3.1-1} that
$d_\gamma\left(0,\sd^F\psi_p{(x',y')}\right)<\al$, which contradicts \eqref{P5-1}.
\end{proof}

\begin{remark}\label{R3.01}
Condition \eqref{P5-1} with the \Fr\ subdifferentials is obviously weaker (hence, more efficient) than its version with the Clarke ones.
However, it is only applicable in Asplund spaces.
\sloppy
\end{remark}

The key condition \eqref{P5-1} in \cref{P5} involves subdifferentials of the function $\psi_p$.
Subgradients of this function belong to $X^*\times Y^*$ and have two component vectors $x^*$ and $y^*$.
In view  of the  representation \eqref{dnorm} of the dual norm on $X^*\times Y^*$, the contributions of the vectors $x^*$ and $y^*$ to the condition \eqref{P5-1} are different.
The next corollary exposes this difference.

\begin{corollary}\label{C7}
%Let $\delta\in]0,+\infty]$ and $\mu\in]0,+\infty]$.
If
there exists an $\eps>0$ such that
$\|x^*\|\ge\al$
for all $p\in P$, $x\in B_{\de+\mu}(\bx)$ and $y\in Y$ satisfying \eqref{P4.1-1}, and all
$(x^*,y^*)\in \sd^C\psi_p(x,y)$ with $\|y^*\|<\eps$;
particularly~if
\begin{align*}%\label{C8-1}
\liminf_{\substack{F_p\iv(\by)\not\ni x\to\bx,\,
F(p,x)\ni y\to\by,\,y^*\to0\\
p\in P,\,y\ne\by,\, (x^*,y^*)\in\sd^C\psi_p(x,y)}}\|x^*\|>\al,
\end{align*}
then $F$ is $\al-$subregular in $x$ uniformly in $p$ over $P$ at $(\bx,\by)$ with $\delta$ and $\mu$.

If $X$ and $Y$ are Asplund, then the above assertion is valid with $\sd^F$ in place of $\sd^C$.
\end{corollary}

\begin{proof}
Suppose $F$ is not $\al-$subregular in $x$ uniformly in $p$ over $P$ at $(\bx,\by)$ with $\delta$ and $\mu$.
Let $\eps>0$ and $\ga:=\eps/\al$.
By \cref{P5}, there exist $p\in P$, $x\in B_{\de+\mu}(\bx)$ and $y\in Y$ satisfying \eqref{P4.1-1}, and a subgradient $(x^*,y^*)\in \sd^{C}\psi_p(x,y)$ ($(x^*,y^*)\in \sd^{F}\psi_p(x,y)$ if $X$ and $Y$ are Asplund)
such that $\|(x^*,y^*)\|_{\ga}<\al$.
In view of the representation of the dual norm \eqref{dnorm}, this implies $\|x^*\|<\al$ and $\|y^*\|<\al\ga=\eps$, a contradiction.

\end{proof}

The function $\psi_p$ involved in the subdifferential sufficient conditions for the uniform $\al-$subregularity in \cref{P5}, is itself a sum of two functions.
We are now going to apply the sum rules again to obtain sufficient conditions in terms of Clarke and Fr\'echet normals to $\gph F_p$.

\begin{theorem}\label{T2}
%Let $\delta\in]0,+\infty]$ and $\mu\in]0,+\infty]$.
The mapping $F$ is $\al-$subregular in $x$ uniformly in $p$ over $P$ at $(\bx,\by)$ with $\delta$ and $\mu$ if, for some $ \ga>0$, and all $p\in P$, $x\in B_{\de+\mu}(\bx)$ and $y\in Y$ satisfying \eqref{P4.1-1}, one of the following conditions is satisfied:
\begin{enumerate}
\item
{with $N:=N^C$,}
\begin{align}\label{T2-1}
d_\ga((0,-y^*),N_{\gph F_p}(x,y))\ge\al
\end{align}
for all $y^*\in Y^*$ satisfying
\begin{align}\label{T1-1}
\|y^*\|=1,\quad\langle y^*,y-\by\rangle=\|y-\by\|;
\end{align}

\item
$X$ and $Y$ are Asplund, and there exists a $\tau\in]0,1[$ such that inequality \eqref{T2-1} holds with $N:=N^F$ for all  $y^*\in Y^*$ satisfying
\begin{align}\label{T1-2}
\|y^*\|=1,\quad \langle y^*,y-\by\rangle>\tau\|y-\by\|.
\end{align}
\end{enumerate}
\end{theorem}

\if{
\begin{proof}
Recall from \eqref{psi} that $\psi_p$ is a sum of two functions: the Lipschitz continuous convex function $v\mapsto g(v):=\|v-\by\|$ and the indicator function of the closed set $\gph F_p$.
\begin{enumerate}
\item
By \cref{SR}(ii), $\sd^C\psi_p(x,y)\subset\{0\}\times\sd g(y) +N^C_{\gph F_p}(x,y)$.
Since $y\ne\by$,
by \cref{L3}(ii), $\sd g(y)$ is a set of all $y^*\in Y^*$ satisfying \eqref{T1-1}.
Hence, condition (i) implies \eqref{P5-1} with $\sd:=\sd^C$.
\sloppy

\item
Let $X$ and $Y$ be Asplund, and $\tau\in]0,1[$.
By \cref{SR}(iii), if $(u^*,v^*)\in\sd^F\psi_p(x,y)$, then, for any $\varepsilon>0$, there exist $x_1\in B_\varepsilon({x})$, $y_1,y_2\in B_\varepsilon(y)$ with $(x_1,y_1)\in\gph F_p$, such that $(u^*,v^*)\in\{0\}\times\sd g(y_2) +N^F_{\gph F_p}(x_1,y_1)+\eps\B_{X^*\times Y^*}$.
\end{enumerate}
\end{proof}
}\fi

\begin{proof}
Suppose $F$ is not $\al-$subregular in $x$ uniformly in $p$ over $P$ at $(\bx,\by)$ with $\delta$ and $\mu$.
Let $\ga>0$.
In view of \cref{P5},
there exist $p\in P$, $x\in B_{\de+\mu}(\bx)$ and $y\in Y$ satisfying \eqref{P4.1-1}, and a subgradient
$(\hat x^*,\hat y^*)\in\sd\psi_p(x,y)$ such that
$\|(\hat x^*,\hat y^*)\|_{\ga}<\al$, where either $\sd:=\sd^C$ (if $X$ and $Y$ are general Banach spaces) or ${\sd:=\sd^F}$ (if $X$ and $Y$ are Asplund).
Recall from \eqref{psi} that $\psi_p$ is a sum of two functions: the Lipschitz continuous convex function $v\mapsto g(v):=\|v-\by\|$ and the indicator function of the closed set $\gph F_p$.
\begin{enumerate}
\item
By \cref{SR}(ii), there exist  $y^*\in \sd g(y)$ and $(u^*,v^*)\in N^C_{\gph F_p}(x,y)$  such that
$(\hat x^*,\hat y^*)=(0,y^*)+(u^*,v^*)$.
Thus,
\begin{align*}
d_\ga((0,-y^*),N^C_{\gph F_p} (x,y))\le\|(0,y^*)+(u^*,v^*)\|_\ga=\|(\hat x^*,\hat y^*)\|_{\ga}<\al,
\end{align*}
which contradicts \eqref{T2-1}.
Since $y\ne\by$,
by \cref{L3}, $y^*$ satisfies conditions \eqref{T1-1}.

\item
Let $X$ and $Y$ be Asplund, and $\tau\in]0,1[$.
By \cref{SR}(iii), for any $\varepsilon>0$, there exist $x_1\in B_\varepsilon({x})$, $y_1,y_2\in B_\varepsilon(y)$ with $(x_1,y_1)\in\gph F_p$, and $y^*\in\sd g(y_2)$, $(u^*,v^*)\in N^F_{\gph F_p}(x_1,y_1)$ such that
\begin{align}\label{T2-7}
\|(0,y^*)+(u^*,v^*)-(\hat x^*,\hat y^*)\|_{\ga}<\eps.
\end{align}
The number $\eps$ can be chosen small enough to ensure that $x_1\in B_{\de+\mu}(\bx)\setminus F_p\iv(\by)$,
$y_1\in B_{\al\mu}(\by)$, $y_2\ne\by$, and
\begin{gather*}
\|y_1-\by\|\ge\frac{1}{2}\|y-\by\|,\;
\|y_2-y_1\|<\frac{1-\tau}{4}\|y-\by\|,\;
\|(\hat x^*,\hat y^*)\|_{\ga}+\eps<\al.
\end{gather*}
By \cref{L3}, we have
$\|y^*\|=1$ and $\langle y^*,y_2-\by\rangle=\|y_2-\by\|$.
Moreover,
\begin{gather*}
\|y_2-y_1\|<\frac{1-\tau}{4}\|y-\by\|\le\frac{1-\tau}{2} \|y_1-\by\|,
\end{gather*}
and consequently,
\begin{align*}
\langle y^*,y_1-\by\rangle
&\ge\langle y^*,y_2-\by\rangle-\|y_2-y_1\|=\|y_2-\by\|-\|y_2-y_1\|\\
&\ge\|y_1-\by\|-2\|y_2-y_1\|>\tau\|y_1-\by\|.
\end{align*}
Making use of \eqref{T2-7}, we obtain
\begin{align*}
d_\ga((0,-y^*),N^F_{\gph F_p} (x_1,y_1))
\le \|(0, y^*)+(u^*,v^*)\|_\ga
<\|(\hat x^*,\hat y^*)\|_{\ga}+\varepsilon<\al.
\end{align*}
This contradicts \eqref{T2-1}.
\end{enumerate}
\end{proof}

\begin{remark}\label{R3.02}
\begin{enumerate}
\item
Condition \eqref{T2-1} with the \Fr\ normal cones is obviously weaker (hence, more efficient) than its version with the Clarke ones; cf. \cref{R3.01}.
However, the Asplund space sufficient condition for uniform $\al-$subregularity in part (ii) of \cref{T2} is not necessarily weaker than its general Banach space version in part~(i), as it replaces the equality in \eqref{T1-1} with a less restrictive inequality in \eqref{T1-2}, which involves an additional parameter $\tau$.
Of course, $\tau$ can be chosen arbitrarily close to 1 making the difference between the constraints \eqref{T1-1} and \eqref{T1-2} less significant.
The weaker than \eqref{T1-1} conditions \eqref{T1-2} employed in part (ii) of \cref{T2} are due to the approximate subdifferential sum rule (\cref{SR}(iii)) used in its proof.

\item
The following alternative sufficient condition has been established half way within the proof of part (ii) of \cref{T2}:
\smallskip

\emph{$X$ and $Y$ are Asplund, and, given any $\eps>0$, inequality \eqref{T2-1} holds with $N:=N^F$ for all $v\in B_\eps(y)$ and all $y^*\in Y^*$ satisfying \eqref{T1-1} with $v$ in place of $y$.}
\smallskip

It employs the stronger equality conditions \eqref{T1-1} instead of \eqref{T1-2}, but involves an unknown vector $v$ (arbitrarily close to $y$).
Conditions of this type are used by some authors, but we prefer more explicit ones in \cref{T2}(ii) and the statements derived from it.
\end{enumerate}
\end{remark}

The qualitative sufficient conditions for uniform regularity follow immediately.

\begin{corollary}\label{C5.3}
The mapping $F$ is subregular in $x$ uniformly in $p$ over $P$ at $(\bx,\by)$ if one of the following conditions is satisfied:
\begin{enumerate}
\item
$\ds\sup_{\ga>0}\liminf_{\substack{F_p\iv(\by)\not\ni x\to\bx,\,
F(p,x)\ni y\to\by\\p\in P,\,y\ne\by,\,\|y^*\|=1,\,
\langle y^*,y-\by\rangle =\|y-\by\|}} d_\ga((0,-y^*),N^C_{\gph F_p}(x,y))>0$;
\smallskip
\item
$X$ and $Y$ are Asplund, and\\
$\ds\sup_{\ga>0,\tau\in]0,1[} \liminf_{\substack{F_p\iv(\by)\not\ni x\to\bx,\,
F(p,x)\ni y\to\by\\p\in P,\,y\ne\by,\,\|y^*\|=1,\,
\langle y^*,y-\by\rangle>\tau\|y-\by\|}}d_\ga((0,-y^*),N^F_{\gph F_p} (x,y))>0$.
\end{enumerate}
\end{corollary}
\if{
\begin{proof}
We provide the proof for item (i).
The one of (ii) is analogous.
Suppose condition \eqref{C9-2} is satisfied, i.e. there exist $\ga,\al,\eta\in]0,+\infty]$ such that for all $p\in B_\eta(\bar p),x\in B_{\eta}(\bx)\setminus F_p\iv(\by)$, $y\in F(p,x)$ with $0<\|y\|<\eta$, and all $y^*\in Y^*$ satisfying
\eqref{T1-1}, it holds
\begin{align}\label{T2-4}
d_\ga((0,-y^*),N^C_{\gph F_p}(x,y))>\al,
\end{align}
and $\gph F_p\cap[B_{\eta}(\bx)\times B_{\eta}(0)]$ is closed for all $p\in B_{\eta}(\bar p)$.
Choose $\mu,\delta\in]0,+\infty]$ such that $\max\{\de,\al\mu,\de+\mu\}:=\eta$,
inequality \eqref{T2-4} holds for all
$p\in B_\de(\bar p),x\in B_{\de}(\bx)\setminus F_p\iv(\by)$, $y\in F(p,x)$ with $0<\|y\|<\al\mu$, and
$\gph F_p\cap[B_{\de+\mu}(\bx)\times B_{\al\mu}(0)]$ is closed for all $p\in B_{\de}(\bar p)$.
In view of \cref{T2}, $F$ is $\al-$subregular at $\bx$ with $\de$ and $\mu$.
\end{proof}
}\fi

\if{
The conditions in \cref{C5.3} employs two assumptions: the closedness of $\gph F_p$, and either condition (i) or condition (ii).
The next example shows that either condition (i) or condition (ii) cannot be dropped.
}\fi

The next example illustrates the sufficient conditions for subregularity in
\cref{C5.3}.

\begin{example}
Let $P=X=Y:=\R$,
$F(p,x):=\{(p-x)^2\}$
for all $p\in P$ and $x\in X$,
and let $\by:=0$.
By \eqref{G}, $F_p\iv(\by)=\{p\}$.
Thus, $d(x,F_p\iv(\by))=|x-p|$
and $d(\by,F(p,x))=(x-p)^2$
for all $p\in P$ and $x\in X$.
Hence, for any $\al>0$ and $p\in P$, inequality \eqref{D1.2-2} is violated when $x$ sufficiently close to $\bx:=0$, i.e.
the mapping $F$ is not {subregular in $x$ uniformly in $p$ over $P$ at $(\bx,\by)$}.
Observe that
$\gph F_p$ is closed for all $p\in P$, and, for any $p\in P$ and $(x,y)\in\gph F_p$,
\begin{align*}
(2(x-p),-1)\in N^C_{\gph F_p}(x,y)=N^F_{\gph F_p}(x,y).
\end{align*}
Let $p=0$, $x\ne0$, $y=x^2$, and $y^*\in\R$ satisfy \eqref{T1-1} {or \eqref{T1-2}}, hence, $y^*=1$, and, for any $\ga>0$, $d_\ga((0,-y^*),(2x,-1))=2|x|\to0$ as $x\downarrow0$.
Both inequalities in \cref{C5.3} are not satisfied.
\end{example}

\cref{T2} yields sufficient conditions for uniform $\al-$subregularity in terms of coderivatives.

\begin{corollary}\label{C3.3}
%Let $\delta\in]0,+\infty]$ and $\mu\in]0,+\infty]$.
The mapping $F$ is $\al-$subregular in $x$ uniformly in $p$ over $P$ at $(\bx,\by)$ with $\delta$ and $\mu$ if, for some $\eta\in]0,+\infty]$, and all $p\in P$, $x\in B_{\de+\mu}(\bx)$ and $y\in Y$ satisfying \eqref{P4.1-1}, one of the following conditions is satisfied:
\begin{enumerate}
\item
with $D^*:=D^*_C$, for all $y^*\in Y^*$ satisfying \eqref{T1-1},
it holds
\begin{align}\label{C3.3-1}
d(0,D^*F_p(x,y)(B_{\eta}(y^*)))\ge\al;
\end{align}

\item
$X$ and $Y$ are Asplund, and there exists a $\tau\in]0,1[$ such that inequality \eqref{C3.3-1} holds with $D^*:=D^*_F$ for all $y^*\in Y^*$ satisfying \eqref{T1-2}.
\end{enumerate}
\end{corollary}

\begin{proof}
Given an $\eta\in]0,+\infty]$, set $\ga:=\al\iv\eta$.
In view of the representations \eqref{coder} of the coderivative and \eqref{dnorm} of the dual norm, condition \eqref{T2-1} means that $\|u^*\|+\ga\iv{\|v^*-y^*\|}\ge\al$ for all $v^*\in Y^*$ and ${u^*\in D^*F_p(x,y)(v^*)}$.
The last inequality is obviously satisfied if either $\|u^*\|\ge\al$ or $\|v^*-y^*\|\ge\eta$, or equivalently, if $\|u^*\|\ge\al$ when $v^*\in B_{\eta}(y^*)$.
\sloppy
\end{proof}

The coderivative sufficient condition \eqref{C3.3-1} can be replaced by its `normalized' (and a little stronger!) version.

\begin{corollary}\label{C3.4}
%Let $\delta\in]0,+\infty]$ and $\mu\in]0,+\infty]$.
The mapping $F$ is $\al-$subregular in $x$ uniformly in $p$ over $P$ at $(\bx,\by)$ with $\delta$ and $\mu$ if, for some $\eta\in]0,1[$, and all $p\in P$, $x\in B_{\de+\mu}(\bx)$ and $y\in Y$ satisfying \eqref{P4.1-1}, one of the following conditions is satisfied:
\begin{enumerate}
\item
{with $D^*:=D^*_C$},
\begin{align}\label{C3.4-1}
d\left(0,D^*F_p(x,y)\left(\frac{v^*}{\|v^*\|}\right) \right) \ge\frac{\al}{1-\eta}
\end{align}
for all $y^*\in Y^*$ satisfying \eqref{T1-1} and $v^*\in B_{\eta}(y^*)$;

\item
$X$ and $Y$ are Asplund, and there exists a $\tau\in]0,1[$ such that inequality \eqref{C3.4-1} holds with $D^*:=D^*_F$ for all $y^*\in Y^*$ satisfying \eqref{T1-2} and $v^*\in B_{\eta}(y^*)$.
\end{enumerate}
\end{corollary}

\begin{proof}
Let $\eta\in]0,1[$, $p\in P$, $x\in B_{\de+\mu}(\bx)$ and $y\in Y$ satisfy \eqref{P4.1-1}, and $y^*\in Y^*$ satisfy either \eqref{T1-1} or \eqref{T1-2}.
We need to show that, if inequality \eqref{C3.4-1} holds for all $v^*\in B_{\eta}(y^*)$, then inequality \eqref{C3.3-1} holds.
First note that, in view of \eqref{T1-1} or \eqref{T1-2}, $\|y^*\|=1$.
Let $v^*\in B_{\eta}(y^*)$ and $u^*\in D^*F_p(x,y)(v^*)$.
Then $\|v^*\|>1-\eta\in]0,+\infty]$.
Thus, condition \eqref{C3.4-1} is well defined.
Moreover, $u^*/\|v^*\|\in D^*F_p(x,y)(v^*/\|v^*\|)$ and, {in view of \eqref{C3.4-1}}, $\|u^*\|\ge\al\|v^*\|/(1-\eta)>\al$, i.e. inequality \eqref{C3.3-1} holds.
\end{proof}

The next qualitative assertion is an immediate consequence of \cref{C3.3}.

\begin{corollary}\label{C3.7}
The mapping $F$ is subregular in $x$ uniformly in $p$ over $P$ at $(\bx,\by)$ if one of the following conditions is satisfied:
\begin{enumerate}
\item
$\ds\lim_{\de\downarrow0}\;
\inf_{\substack{x\in B_{\de}(\bx)\setminus F_p\iv(\by),\,
\by\ne y\in F(p,x)\cap B_{\de}(\by)\\p\in P,\, \|y^*\|=1,\,\langle y^*,y-\by\rangle=\|y-\by\|}} d(0,D^*_CF_p(x,y)(B_{\de}(y^*)))>0;$
\smallskip
\item
$X$ and $Y$ are Asplund, and\\
$\ds\lim_{\de\downarrow0,\,\tau\uparrow1}\;
\inf_{\substack{x\in B_{\de}(\bx)\setminus F_p\iv(\by),\,
\by\ne y\in F(p,x)\cap B_{\de}(\by)\\ p\in P,\, \|y^*\|=1,\,\langle y^*,y-\by\rangle>\tau\|y-\by\|}} d(0,D^*_FF_p(x,y)(B_{\de}(y^*)))>0.$
\end{enumerate}
\end{corollary}

\begin{remark}\label{R3.1}
\begin{enumerate}
\item
In the case when $P$ is a neighborhood of a given point $\bar p$ in a metric space,
\cref{C3.7} (taking into account  \cref{R3.02}(ii) in some instances) improves
\cite[Theorem~3.6]{LedZhu99}, \cite[Theorem~3.4]{NgaThe04},
\cite[Theorem~3.2]{LeeTamYen08},
\cite[Theorem~3.5]{HuyYao09},
\cite[Theorem~3.1]{HuyKimNin12},
\cite[Corollary~2.2]{Ngh14},
\cite[Theorem~1]{ChuKim16} (in the linear setting),
and \cite[Theorem~4.1(e)]{Iof17.1}.
\if{
\NDC{5.1.20
Should we explain in details about the improvements?}
}\fi
\item
Clarke
%and Fr\'echet
normal cones in this section can be replaced by Ioffe's \emph{$G$-nor\-mal cones}~\cite{Iof17}.
%, corresponding to the \emph{approximate $G$-sub\-differentials}, which, like the Clarke ones, possess an exact sum rule in general Banach spaces; cf. \cite[Theorem~4.69]{Iof17}.
\end{enumerate}
\end{remark}

\subsection{Dual Necessary Conditions}\label{S4.2}

In this subsection, $X$ and $Y$ are normed spaces, $F:P\times X\rightrightarrows Y$, $\bx\in X$, $\by\in Y$,
$\al>0$, ${\delta\in]0,+\infty]}$, $\mu\in]0,+\infty]$, and we assume that $\gph F_p$ is convex for all $p\in P$.

The next statement provides a necessary condition for uniform $\al-$subregularity in terms of subdifferentials of the function $\psi_p$ defined by \eqref{psi}.

\begin{proposition}\label{P5.1}
If $F$ is $\al-$subregular in $x$ uniformly in $p$ over $P$ at $(\bx,\by)$ with $\delta$ and $\mu$, then inequality \eqref{P5-1} is satisfied with
$\ga:=\al\iv$ for all $p\in P$, $x\in B_{\de}(\bx)$ and $y\in Y$ satisfying~\eqref{P4.1-1}.
\end{proposition}

\begin{proof}
Under the assumptions made, the function $\psi_p$ is convex for all $p\in P$.
Let ${\ga:=\al\iv}$, and
$p\in P$, $x\in B_{\de}(\bx)$ and $y\in Y$ satisfy \eqref{P4.1-1}.
For any $(x^*,y^*)\in\sd\psi_p(x,y)$, we have
\begin{align*}
\|(x^*,y^*)\|_{\ga}
&=\sup_{\substack{(u,v)\ne(0,0)}} \dfrac{\ang{(x^*,y^*),(u,v)}} {\|(u,v)\|_{\ga}}\\
&=\limsup_{\substack{u{\to}x,\, v\to y\\
(u,v)\ne(x,y)}} \dfrac{-\ang{(x^*,y^*),(u,v) -
(x,y)}}{\|(u,v)-(x,y)\|_{\ga}}\\
&\ge\limsup_{\substack{u{\to}x,\, v\to y\\ (u,v)\ne(x,y)}} \dfrac{\psi_p(x,y)-\psi_p(u,v)} {\|(u,v)-(x,y)\|_{\ga}}\\
&=\limsup_{\substack{u\to x,\,v\to y\\(u,v)\in\gph F_p,\,(u,v)\ne(x,y)}}
\dfrac{\|y-\by\|-\|v-\by\|}{\|(u,v)-(x,y)\|_\ga}.
\end{align*}
By \cref{C2.2}(i), we have
$\|(x^*,y^*)\|_{\ga}\ge\al.$
Taking the infimum in the \LHS\ of the last inequality over $(x^*,y^*)\in\sd\psi_p(x,y)$, we obtain inequality \eqref{P5-1}.
\end{proof}

Combining the above statement with \cref{P5}, we obtain a complete subdifferential characterization of the uniform subregularity in the convex setting.

\begin{corollary}%\label{C3.1}
Let $X$ and $Y$ be Banach, and $\gph F_p$ be closed for all ${p\in P}$.
%The mapping $F$ is $\al-$subregular in $x$ uniformly in $p$ over $P$ at $(\bx,\by)$ with ${\delta}$ and $\mu$ if and only if inequality \eqref{P5-1} holds with
%$\ga:=\al\iv$ for all $p\in P$, $x\in B_{\de}(\bx)$ and $y\in Y$ %satisfying~\eqref{P4.1-1}.
%As a consequence, $F$ is subregular in $x$ uniformly in $p$ over $P$ at $(\bx,\by)$ if and only~if
%\begin{align*}%\label{C5.1-1}
%\sup_{\ga>0}
%\liminf_{\substack{x\to\bx,\,
%y\to\by\\p\in P,\,x\notin F_p\iv(\by),\,\by\ne y\in F(p,x)}}
%d_\gamma\left(0,\sd\psi_p(x,y)\right)>0.
%\end{align*}
The mapping $F$ is subregular in $x$ uniformly in $p$ over $P$ at $(\bx,\by)$ if and only if
\begin{align*}%\label{C5.1-1}
\sup_{\ga>0}
\liminf_{\substack{x\to\bx,\,
y\to\by\\p\in P,\,x\notin F_p\iv(\by),\,\by\ne y\in F(p,x)}}
d_\gamma\left(0,\sd\psi_p(x,y)\right)>0.
\end{align*}
\end{corollary}

\if{
\begin{remark}
In view of the
convexity of $\psi_p$ and \cref{P1.3}(ii), we have
$|\nabla \psi_p|_\ga(x,y)=d_\gamma\left(0,\sd\psi_p(x,y)\right)$.
Hence, \cref{C3.1} can be seen as the dual counterpart of \cref{T4.1}.
\end{remark}
}\fi

The next corollary follows from \cref{P5.1} in view of the representation \eqref{dnorm} of the dual norm.

\begin{corollary}\label{C7.1}
If $F$ is $\al-$subregular in $x$ uniformly in $p$ over $P$ at $(\bx,\by)$ with $\delta$ and ${\mu}$, then
$\|x^*\|\ge\al(1-\|y^*\|)$
for all $p\in P$, $x\in B_{\de}(\bx)$ and $y\in Y$ satisfying \eqref{P4.1-1}, and all $(x^*,y^*)\in \sd\psi_p(x,y)$.
As a consequence,
\begin{align*}
\liminf_{\substack{F_p\iv(\by)\not\ni x\to\bx,\,F(p,x)\ni y\to \by,\,y^*\to0\\p\in P,\,y\ne\by,\, (x^*,y^*)\in\sd\psi_p(x,y)}} \|x^*\|\ge\al.
\end{align*}
\end{corollary}

The next statement gives a partial converse to \cref{T2}.

\begin{theorem}\label{C3.08}
If $F$ is $\al-$subregular in $x$ uniformly in $p$ over $P$ at $(\bx,\by)$ with $\delta$ and $\mu$, then for all $p\in P$, $x\in B_{\de}(\bx)$ and $y\in Y$ satisfying \eqref{P4.1-1}, and all $y^*\in Y^*$ satisfying \eqref{T1-1},
inequality \eqref{T2-1} is satisfied with $\ga:=\al\iv$.
\end{theorem}

\begin{proof}
Observe that $\psi_p$ is the sum of the convex continuous function $v\mapsto g(v):={\|v-\by\|}$ and the indicator function of the convex set $\gph F_p$.
By \cref{SR}(i), $\sd\psi_p(x,y)=\{0\}\times\sd g(y)+N_{\gph F_p}(x,y)$.
The assertion follows from \cref{P5.1} in view of \cref{L3}(ii).
\sloppy
\end{proof}

Combining \cref{T2,C3.08},
we can formulate a necessary and sufficient characterization of the uniform subregularity in the convex setting.

\begin{corollary}\label{C3.8}
Let $X$ and $Y$ be Banach, and $\gph F_{p}$ be closed for all ${p\in P}$.
%The mapping $F$ is $\al-$subregular in $x$ uniformly in $p$ over $P$ at $(\bx,\by)$ with ${\delta}$ and $\mu$ if and only if
%for all $p\in P$, $x\in B_{\de}(\bx)$ and $y\in Y$ satisfying \eqref{P4.1-1}, and all $y^*\in Y^*$ satisfying \eqref{T1-1},
%inequality \eqref{T2-1} is satisfied with $\ga:=\al\iv$.
%As a consequence, $F$ is subregular in $x$ uniformly in $p$ over $P$ at $(\bx,\by)$ if and only~if
%\begin{align}\label{C3.9-1}
%\sup_{\ga>0}\liminf_{\substack{F_p\iv(\by)\not\ni x\to\bx,\,
%F(p,x)\ni y\to\by\\p\in P,\,y\ne\by,\,\|y^*\|=1,\,
%\langle y^*,y-\by\rangle=\|y-\by\|}} d_\ga((0,-y^*),N_{\gph F_p}(x,y))>0.
%\end{align}
The mapping $F$ is subregular in $x$ uniformly in $p$ over $P$ at $(\bx,\by)$ if and only if
 \begin{align}\label{C3.9-1}
\sup_{\ga>0}\liminf_{\substack{F_p\iv(\by)\not\ni x\to\bx,\,
F(p,x)\ni y\to\by\\p\in P,\,y\ne\by,\,\|y^*\|=1,\,
\langle y^*,y-\by\rangle=\|y-\by\|}} d_\ga((0,-y^*),N_{\gph F_p}(x,y))>0.
\end{align}
\end{corollary}

{The next example illustrates the necessary condition in \cref{C3.08}}.

\begin{example}
Let $P=X=Y:=\R$,
$F(p,x):=\{p-x\}$
for all $p\in P$ and $x\in X$,and let $\bx=\by:=0$.
By \eqref{G}, $F_p\iv(\by)=\{p\}$.
Thus, $d(x,F_p\iv(\by))=d(\by,F(p,x))=|x-p|$
for all $p\in P$ and $x\in X$.
Hence, inequality \eqref{D1.2-2} is satisfied for all $p\in P$, $x\in X$, and $\al\in]0,1]$, i.e.
the mapping $F$ is $\al-$subregular in $x$ uniformly in $p$ over $P$ at $(\bx,\by)$ for any $\al\in]0,1]$.
{We have}
$\gph F_p=\{(x,y)\mid y=p-x\}$ is closed and convex for all $p\in P$, and $N_{\gph F_p}(x,y)=\{(t,t)\mid t\in\R\}$ for any $(x,y)\in\gph F_p$.
Let $y^*\in\R$ satisfy \eqref{T1-1}.
Then $y^*=1$ if $y>0$, and $y^*=-1$ if $y<0$.
It is easy to check that, given a $\ga>0$, in both cases the distance
$d_\ga((0,-y^*),N_{\gph F_p}(x,y))$ equals 1 if $\ga\le1$, or $\ga\iv$ if $\ga>1$.
Hence, condition \eqref{C3.9-1} is satisfied, confirming the uniform {subregularity} of $F$.
\end{example}

The next statement is a consequence of
\cref{C3.08}.
It is in a sense a partial converse to \cref{C3.3}.

\begin{corollary}\label{C4.9}
If $F$ is $\al-$subregular in $x$ uniformly in $p$ over $P$ at $(\bx,\by)$ with ${\delta}$ and $\mu$, then $d(0,D^*F_p(x,y)(B_{\eta}(y^*)))\ge\al(1-\eta)$
for all ${\eta\in]0,1[}$, $p\in P$, $x\in B_{\de}(\bx)$ and $y\in Y$ satisfying \eqref{P4.1-1},
and all $y^*\in Y^*$ satisfying \eqref{T1-1}.
\sloppy
\end{corollary}

\begin{proof}
%Given an $\eta\in]0,+\infty]$, set $\ga:=\al\iv\eta$.
In view of the representations \eqref{coder} of the coderivative and \eqref{dnorm} of the dual norm, condition \eqref{T2-1} with $\ga:=\al\iv$ means that $\|u^*\|+\al{\|v^*-y^*\|}\ge\al$ for all $v^*\in Y^*$ and ${u^*\in D^*F_p(x,y)(v^*)}$.
Hence, it yields $\|u^*\|>\al(1-\eta)$ if
$\|v^*-y^*\|<\eta$.
\sloppy
\end{proof}

Combining the above statement with \cref{C3.3}, we obtain a {complete} coderivative characterization of the uniform $\al-$subregularity in the convex setting.
It improves \cite[Theorem~3]{ChuKim16} (in the linear case).

\begin{corollary}%\label{C3.8}
Let $X$ and $Y$ be Banach, and $\gph F_{p}$ be closed for all $p\in P$.
The mapping $F$ is $\al-$subregular in $x$ uniformly in $p$ over $P$ at $(\bx,\by)$ if and only~if
$$
\lim_{\de\downarrow0}\;\inf_{\substack{x\in B_{\de}(\bx)\setminus F_p\iv(\by),\,
\by\ne y\in F(p,x)\cap B_{\de}(\by)\\ p\in P,\,\|y^*\|=1,\, \langle y^*,y-\by\rangle=\|y-\by\|}}
d(0,D^*F_p(x,y)(B_{\de}(y^*)))\ge\al.
$$
\end{corollary}

\if{
The next corollary provides quantitative dual characterizations of Robinson subregularity.
\begin{corollary}\label{T4.2}
Let $X$ and $Y$ be normed spaces, $\al>0$, $\delta\in]0,+\infty]$ and $\mu\in]0,+\infty]$.
Suppose $\gph F_{\bar p}$ is convex, and the mapping $F$ is Robinson $\al-$subregular at $\bx$ with $\de$ and $\mu$.
The following statements hold.
\begin{enumerate}
\item
Inequality \eqref{T4-2} holds for $\ga:=\al\iv$ and all $x,y$ satisfying \eqref{P4.2-1}, and $y^*\in Y^*$ satisfying \eqref{T1-1}.
\item
Let $\al<\frac{1}{\sqrt{2}}$.
Inequality \eqref{T3.3-2} holds for some $\ga>0$ and all $x,y$ satisfying \eqref{P4.2-1}, $y^*\in J_\ga(y)$, and $x^*\in D^*F_{\bar p}(x,y)(y^*)$.
\end{enumerate}
\end{corollary}
\begin{proof}
Cf. the proof of (??).
\end{proof}
}\fi
\if{
Thanks to Corollaries~\ref{T4}, \ref{T5} and \ref{T4.2}, we can obtain complete dual characterizations of Robinson subregularity.
\begin{corollary}\label{C3.9}
Let $X$ and $Y$ be Banach, $\gph F_{\bar p}$ be closed and convex.
The mapping $F$ is Robinson subregular at $(\bar p,\bx)$ if and only if one of the following conditions is satisfied:
\begin{enumerate}
\item
\begin{align*}
\sup_{\ga>0}\liminf_{\substack{x\to\bx,\,
y\to0,\,x\notin G(\bar p)\\y\in F_{\bar p}(x)\setminus\{0\},\,\|y^*\|=1,\,
\langle y^*,y\rangle=\|y\|}}d_\ga((0,-y^*),N_{\gph F_{\bar p}}(x,y))>0;
\end{align*}
\item
\begin{align*}
\liminf_{\lambda\downarrow0}
\big\{\|x^*\|\mid x^*\in D^*F_{\bar p}(x)(y^*),\,
p\in B_\lambda(\bar p),\,x\in B_\lambda(\bx)\setminus G(\bar p),\\
y\in F_{\bar p}(x)\cap (B_{\lambda}(0)\setminus\{0\}),\,
y^*\in J_\lambda(y)\big\}>0.
\end{align*}
\end{enumerate}
\end{corollary}
\begin{remark}
The characterization in item (i) is new, while the one in item (ii) recaptures \cite[Theorem~4]{ChuKim16}.
\end{remark}
}\fi

\section{Metric Subregularity, Metric Regularity and Implicit Multifunctions}\label{S6}

In this section, we illustrate the necessary and sufficient conditions for the uniform subregularity established in the preceding sections by characterizing several conventional properties of \SVM s.

\subsection{Metric Subregularity}

As observed in the Introduction, the conventional regularity properties in \cref{D1.1} are particular cases of the uniform regularity properties in \cref{D1.2} corresponding to $P$ being a singleton, which practically means that the \SVM\ $F$ does not involve a parameter.

The next three statements, which are immediate consequences of the corresponding `parametric' ones in \cref{S4,S5}, illustrate this observation for the case of subregularity.
Here
$X$ and $Y$ are normed spaces, $F:X\rightrightarrows Y$,  $(\bx,\by)\in\gph F$, $\al>0$, $\delta\in]0,+\infty]$ and $\mu\in]0,+\infty]$.

\begin{proposition}\label{P5.2}
\begin{enumerate}
\item
Suppose $\gph F$ is convex.
If $F$ is $\al-$subregular at $(\bx,\by)$ with $\delta$ and $\mu$, then
\begin{align}\label{P5.1-1}
\limsup_{\substack{u\to x,\,v\to y,\,(u,v)\in\gph F,\,(u,v)\ne (x,y)\\
d(u,\bx)<\de+\mu,\,d(v,\by)<\al\mu}}
{\dfrac{\|y-\by\|-\|v-\by\|}{{\|(u-x,v-y)\|_\ga}}}\ge\al
\end{align}
for $\gamma:=\al\iv$, and all
$x\in B_{\de}(\bx)$ and $y\in Y$ satisfying
\begin{align}\label{P5.1-2}
x\notin F\iv(\by),\quad
y\in F(x){\cap B_{\al\mu}(\by)}.
\end{align}

\item
Suppose $X$ and $Y$ are Banach, and $\gph F$ is closed.
If inequality \eqref{P5.1-1} holds for some $\ga>0$, and all $x\in B_{\de+\mu}(\bx)$ and $y\in Y$ satisfying \eqref{P5.1-2},
then $F$ is $\al-$subregu\-lar at $(\bx,\by)$ with  $\delta$ and $\mu$.
\end{enumerate}
\end{proposition}

\begin{proof}
The statement is a consequence of \cref{C2.2}.
\end{proof}

\begin{proposition}%\label{T2}
\begin{enumerate}
\item
Suppose $\gph F$ is convex.
If $F$ is $\al-$subregu\-lar at $(\bx,\by)$ with $\delta$ and $\mu$, then
\begin{align}\label{P5.2-1}
d_\ga((0,-y^*),N_{\gph F}(x,y))\ge\al
\end{align}
for $\ga:=\al\iv$, all $x\in B_{\de}(\bx)$ and $y\in Y$ satisfying \eqref{P5.1-2}, and all $y^*\in Y^*$ satisfying~\eqref{T1-1}.

\item
Suppose $X$ and $Y$ are Banach, and $\gph F$ is closed.
%Let $\delta\in]0,+\infty]$ and $\mu\in]0,+\infty]$.
The mapping $F$ is $\al-$subregu\-lar at $(\bx,\by)$ with $\delta$ and $\mu$ if, for some $ \ga>0$, and all ${x\in B_{\de+\mu}(\bx)}$ and $y\in Y$ satisfying \eqref{P5.1-2}, one of the following conditions is satisfied:
\begin{enumerate}
\item
inequality \eqref{P5.2-1} holds with $N:=N^C$
for all $y^*\in Y^*$ satisfying \eqref{T1-1};
\item
$X$ and $Y$ are Asplund, and there exists a $\tau\in]0,1[$ such that inequality \eqref{P5.2-1} holds with ${N:=N^F}$ for all  $y^*\in Y^*$ satisfying \eqref{T1-2}.
\end{enumerate}
\end{enumerate}
\end{proposition}

\begin{proof}
The statement is a consequence of \cref{T2,C3.08}.
\end{proof}

\begin{proposition}\label{P5.3}
\begin{enumerate}
\item
Suppose $\gph F$ is convex.
If $F$ is $\al-$subregular at $(\bx,\by)$ with $\delta$ and $\mu$, then $d(0,D^*F(x,y)(B_{\eta}(y^*)))\ge\al(1-\eta)$
for any $\eta\in]0,1[$, all $x\in B_{\de}(\bx)$ and $y\in Y$ satisfying \eqref{P5.1-2},
and all $y^*\in Y^*$ satisfying~\eqref{T1-1}.
\sloppy

\item
Suppose $X$ and $Y$ are Banach, and $\gph F$ is closed.
%Let $\delta\in]0,+\infty]$ and $\mu\in]0,+\infty]$.
The mapping $F$ is $\al-$subregular at $(\bx,\by)$ with $\delta$ and $\mu$ if, for some $\eta\in]0,+\infty]$, and all $x\in B_{\de+\mu}(\bx)$ and $y\in Y$ satisfying \eqref{P5.1-2}, one of the following conditions is satisfied:
\begin{enumerate}
\item
with $D^*:=D^*_C$, for all $y^*\in Y^*$ satisfying \eqref{T1-1},
it holds
\begin{align}\label{P5.3-1}
d(0,D^*F(x,y)(B_{\eta}(y^*)))\ge\al;
\end{align}

\item
$X$ and $Y$ are Asplund, and there exists a $\tau\in]0,1[$ such that inequality \eqref{P5.3-1} holds with ${D^*:=D^*_F}$ for all $y^*\in Y^*$ satisfying \eqref{T1-2}.
\end{enumerate}
\end{enumerate}
\end{proposition}

\begin{proof}
The statement is a consequence of \cref{C3.3,C4.9}.
\end{proof}

\begin{remark}
\begin{enumerate}
\item
In \cref{P5.2}(ii), it is sufficient to assume that $X$ and $Y$ are complete metric spaces (with distances in place of norms in condition \eqref{P5.1-1}), or even that $\gph F$ is complete; cf. \cref{R2.1}(iii).
In this setting, the sufficient condition in \cref{P5.2}(ii) can be viewed as a quantitative version of \cite[Corollary~5.8(d)]{Kru15} and \cite[Theorem~2.4(a)]{Iof17.1}.

\item
Proposition \ref{P5.3} improves \cite[Theorem~5.3]{LiMor12}.
In the linear setting, part (ii) of this proposition improves \cite[Theorem~3.3]{LiMor12}, \cite[Theorem~6]{NgaTroThe13}, \cite[Theorem~8]{ChuKim16},
\cite[Theorem~2.6]{Iof17.1},
and the corresponding parts of \cite[Corollary~5.8]{Kru15}.
\cref{P5.3}(ii) with condition (a) recaptures  \cite[Theorem~3.2]{ZheNg10}.
\end{enumerate}	
\end{remark}	

\subsection{Metric Regularity}

The conventional metric regularity is a particular case of the uniform regularity property in \cref{D1.2}(i) corresponding to $P$ being a singleton.
At the same time, as it follows from the observation in \cref{R1.1}, in the normed space setting it can be treated as a particular case of the uniform subregularity property in \cref{D1.2}(ii) for the \SVM\ $\widehat{F}(y,x):={F(x)-y}$, $(y,x)\in Y\times X$ with $y$ considered as a parameter.
Obviously $(\bx,\by)\in\gph F$ if and only if $(\by,\bx,0)\in\gph\widehat{F}$.

The next three statements, which are immediate consequences of the corresponding `parametric' ones in \cref{S4,S5}, illustrate the above observation.
Here
$X$ and $Y$ are normed spaces, $F:X\rightrightarrows Y$,  $(\bx,\by)\in\gph F$, $\al>0$, $\delta\in]0,+\infty]$ and $\mu\in]0,+\infty]$.

\begin{proposition}\label{P5.4}
\begin{enumerate}
\item
Suppose $\gph F$ is convex.
If $F$ is $\al-$regular at $(\bx,\by)$ with $\delta$ and $\mu$, then
\begin{align}\label{P5.4-1}
\limsup_{\substack{u\to x,\,v\to z,\,(u,v)\in\gph F,\,(u,v)\ne (x,z)\\
d(u,\bx)<\de+\mu,\,d(v,y)<\al\mu}}
{\dfrac{\|z-y\|-\|v-y\|}{{\|(u-x,v-z)\|_\ga}}}\ge\al
\end{align}
for $\gamma:=\al\iv$, and all
{$x\in B_\de(\bx)$}, $y\in B_{\de}(\by)$ and $z\in Y$ satisfying
\begin{align}\label{P5.4-2}
{x\notin F\iv(y)},\quad
z\in F(x){\cap B_{\al\mu}(y)}.
\end{align}

\item
Suppose $X$ and $Y$ are Banach spaces, and $\gph F$ is closed.
%Let $\delta\in]0,+\infty]$ and $\mu\in]0,+\infty]$.
The mapping $F$ is $\al-$regular at $(\bx,\by)$ with $\delta$ and $\mu$ if inequality \eqref{P5.4-1}
holds with some $\ga>0$ for all {$x\in B_{\de+\mu}(\bx)$}, $y\in B_{\de}(\by)$ and $z\in Y$ satisfying \eqref{P5.4-2}.
\end{enumerate}
\end{proposition}

\begin{proof}
The statement is a consequence of \cref{C2.2}.
\end{proof}

\if{
\AK{14/06/20.
Should the $\de+\mu$ estimate in the above proposition be applied to $x$ instead of $y$?}
\NDC{17/06/20.
Yes, I think it should be.
I think the condition $y\in B_\de(\by)$ is nothing else but the uniformity in the definition of metric regularity.}
}\fi

\begin{proposition}%\label{C6}
\begin{enumerate}
\item
Suppose $\gph F$ is convex.
If $F$ is $\al-$regular at $(\bx ,\by)$ with $\delta$ and $\mu$, then
\begin{align}\label{T2-5}
d_\ga((0,-y^*),N_{\gph F} (x,z))\ge\al
\end{align}
for $\gamma:=\al\iv$, and all {$x\in B_\de(\bx)$}, $y\in B_{\de}(\by)$ and $z\in Y$ satisfying \eqref{P5.4-2}, and $y^*\in Y^*$ satisfying
\begin{align}\label{T2-8}
\|y^*\|=1,\quad\langle y^*,z-y\rangle=\|z-y\|.
\end{align}

\item
Suppose $X$ and $Y$ are Banach, and $\gph F$ is closed.
%Let $\delta\in]0,+\infty]$ and $\mu\in]0,+\infty]$.
The mapping $F$ is $\al-$regular at $(\bx,\by)$ with  $\delta$ and $\mu$ if,
for some $\ga>0$, and all {$x\in B_{\de+\mu}(\bx)$}, {$y\in B_{\de}(\by)$} and $z\in Y$ satisfying \eqref{P5.4-2}, one of the following conditions holds:
\begin{enumerate}
\item
inequality \eqref{T2-5} holds with $N:=N^C$
for all $y^*\in Y^*$ satisfying
\eqref{T2-8};
\item
$X$ and $Y$ are Asplund, and there exists a $\tau\in]0,1[$ such that inequality \eqref{T2-5} holds with ${N:=N^F}$ for all $y^*\in Y^*$ satisfying
\begin{align}\label{T2-9}
\|y^*\|=1,\quad\langle y^*,z-y\rangle>\tau\|z-y\|.
\end{align}
\end{enumerate}
\end{enumerate}
\end{proposition}

\begin{proof}
The statement is a consequence of \cref{T2,C3.08}.
\end{proof}

\begin{proposition}\label{P5.6}
\begin{enumerate}
\item
Suppose $\gph F$ is convex.
If $F$ is $\al-$regular at $(\bx,\by)$ with $\delta$ and $\mu$, then $d(0,D^*F(x,z)(B_{\eta}(y^*)))\ge\al(1-\eta)$
for all ${\eta\in]0,1[}$, {$x\in B_\de(\bx)$}, $y\in B_{\de}(\by)$ and $z\in Y$ satisfying \eqref{P5.4-2},
and all $y^*\in Y^*$ satisfying \eqref{T2-8}.
\sloppy
\item
Suppose $X$ and $Y$ are Banach, and $\gph F$ is closed.
%Let $\delta\in]0,+\infty]$ and $\mu\in]0,+\infty]$.
The mapping $F$ is $\al-$regular at $(\bx,\by)$ with  $\delta$ and $\mu$ if,
for some $\eta\in]0,+\infty]$ and all {$x\in B_{\de+\mu}\bx)$}, {$y\in B_{\de}(\by)$} and $z\in Y$ satisfying \eqref{P5.4-2},
one of the following conditions holds:
\begin{enumerate}
\item
with $D^*:=D^*_C$, for all $y^*\in Y^*$ satisfying \eqref{T2-8},
it holds
\begin{align}\label{C3.3-3}
d(0,D^*F(x,z)(B_{\eta}(y^*)))\ge\al;
\end{align}
\item
$X$ and $Y$ are Asplund, and there exists a $\tau\in]0,1[$ such that inequality \eqref{C3.3-3} holds with ${D^*:=D^*_F}$ for all $y^*\in Y^*$ satisfying \eqref{T2-9}.
\end{enumerate}
\end{enumerate}
\end{proposition}

\begin{proof}
The statement is a consequence of \cref{C3.3,C4.9}.
\end{proof}

\begin{remark}
\begin{enumerate}
\item
In the normed space setting, the sufficient condition in Propositi\-on~\ref{P5.4}(ii) can be viewed as a quantitative version of \cite[Theorem~1]{Iof00} and \cite[Theorem~2.4(a)]{Iof17.1}.

\item
\cref{P5.6}(ii) enhances \cite[Theorem~3.7]{Chu15}.	
\cref{P5.6}(ii) with condition (a) improves
\cite[Corollary~3.1]{HuyKimNin12} and \cite[Theorem~3.5]{Chu15} (in the linear case), while with condition (b) it
improves (in the linear case) \cite[Theorem~3.1]{Chu15} and \cite[Theorem~7]{ChuKim16}.
\end{enumerate}
\end{remark}

\subsection{Implicit Multifunctions}

Now we get back to the implicit multifunction \eqref{G} and consider its particular case corresponding to the parametric inclusion $\by\in F(p,x)$ (with fixed \LHS), i.e.
\begin{align}\label{G2}
G(p):=\{x\in X\mid \by\in F(p,x)\},\quad p\in P,
\end{align}
where $F:P\times X\rightrightarrows Y$, and $P$, $X$ and $Y$ are metric spaces.
Stability properties of implicit multifunctions, i.e. solution sets of parametric inclusions, are of great importance for many applications and have been the subject of numerous publications; cf., e.g., \cite{Rob75.2,Rob76.2,Bor86,LedZhu99,AzeCorLuc02,KlaKum02, NgaThe04,AzeBen08, LeeTamYen08,HuyYao09,YenYao09,ChiYaoYen10, ChuKruYao11, HuyKimNin12,NgaTroThe13,DonRoc14,Ngh14, Chu15,ChuKim16,GfrOut16.2,Iof17,Iof17.1,Ude21}.
Here, for illustration, we focus on the most well known perturbation stability property of \SVM s called \emph{Aubin property}; cf. \cite{DonRoc14,Mor06.1}.

\begin{definition}%\label{D1.2}
A mapping $G:P\rightrightarrows X$ between metric spaces has the Aubin property at $(\bp,\bx)\in\gph G$ with rate $l>0$ if there exist $\eta\in]0,+\infty]$, $\delta\in]0,+\infty]$ and $\mu\in]0,+\infty]$ such that
\begin{align*}
d(x,G(p))\le l\;d(p,p')
\end{align*}
for all $p,p'\in B_\eta(\bp)$ with $d(p,p')<\mu$, and $x\in G(p')\cap B_\de(\bx)$.
\end{definition}

Similar to Definitions~\ref{D1.1}, \ref{D1.2} and \ref{D1.3}, the inequality $d(p,p')<\mu$ is not essential in the above definition and can be dropped together with the constant $\mu$.
We keep them for consistency with the definitions and characterizations in the preceding sections.
We also establish connections between the constant $\mu$ and the corresponding constants in the other definitions.

Given a point ${(\bp,\bx,\by)\in\gph F}$ and a number $\al>0$, the uniform
$\al-$subregula\-rity property of $F$ at $(\bp,\bx,\by)$ in \cref{D1.3}(ii) means that
there exist $\eta\in]0,+\infty]$, ${\delta\in]0,+\infty]}$ and $\mu\in]0,+\infty]$ such that
\begin{align}\label{43}
\al d(x,G(p))\le d(\by,F(p,x))
\end{align}
for all $p\in B_\eta(\bp)$ and $x\in B_\de(\bx)$ with ${d(\by,F(p,x))}<\al\mu$.
Several primal and dual sufficient and necessary conditions for this property have been formulated in the preceding sections.

Inequality \eqref{43} provides an estimate for the distance from $x$ to the value of the implicit multifunction \eqref{G2} at $p$ in terms of the residual of the parametric inclusion.
However, this estimate does not say much about the behaviour of the implicit multifunction.
An additional assumption on the mapping $F$ is needed, which would allow one to get rid of $F$ in the \RHS\ of the inequality \eqref{43}.
This additional assumption is given in the next definition, which is a modification of the second part of \cite[Definition~3.1]{Iof17.1}, where we borrow the terminology from.
A similar property was considered in \cite{KlaKum02}, where the authors used the name \emph{Lipschitz lower semicontinuity}.

\begin{definition}%\label{D1.2}
Let $l>0$.
The mapping $F$ is said to $l-$recede in $p$ uniformly in $x$ at $(\bp,\bx,\by)$ if there exist $\eta\in]0,+\infty]$, $\delta\in]0,+\infty]$ and $\mu\in]0,+\infty]$ such that
\begin{align}\label{D5.1-1}
d(\by,F(p,x))\le l d(p,p')
\end{align}
for all $x\in B_\de(\bx)$ and $p,p'\in B_\eta(\bp)$ with
$d(p,p')<\mu$ and $\by\in F(p',x)$.
\end{definition}

In what follows we assume that ${(\bp,\bx,\by)\in\gph F}$, $\al>0$, $l>0$, $\eta\in]0,+\infty]$, $\delta\in]0,+\infty]$ and $\mu\in]0,+\infty]$.
The next statement is a modification of \cite[Theorem~3.2]{Iof17.1}.

\begin{proposition}\label{P5.7}
%Let $\al>0$ and $l>0$.
Suppose that $F$
\begin{itemize}
\item
is $\al-$subregular in $x$ uniformly in $p$ at $(\bp,\bx,\by)$ with $\eta$, $\delta$ and $\mu$;
\item
$l-$recedes in $p$ uniformly in $x$ at $(\bp,\bx,\by)$ with $\eta$, $\de$ and $\mu':=\al\mu/l$.
\end{itemize}
Then the mapping $G$ given by \eqref{G2} has the Aubin property at $(\bp,\bx)$ with rate $l/\al$, and $\eta$, $\de$ and $\mu'$.
\end{proposition}

\begin{proof}
Let $p,p'\in B_\eta(\bp)$ with $d(p,p')<\mu'$, and $x\in G(p')\cap B_\de(\bx)$.
By \eqref{G2} and \eqref{D5.1-1}, ${\by\in F(p',x)}$ and
${d(\by,F(p,x))}<\al\mu$.
Using successively \eqref{43} and \eqref{D5.1-1}, we obtain
\sloppy
\begin{align*}%\label{P5.7-1}
d(x,G(p))\le\frac{1}{\al} d(\by,F(p,x))\le\frac{l}{\al}d(p,p').
\end{align*}
The proof is completed.
\end{proof}

Combining \cref{P5.7} with the sufficient conditions for the uniform subregularity formulated in the preceding sections, we can immediately obtain various sufficient conditions for the Aubin property of the implicit multifunction \eqref{G2}.
The next proposition collects three sufficient conditions arising from {\cref{C2.2}(ii)}, \cref{T2} and \cref{C3.3}, respectively.

\begin{proposition}\label{P5.8}
Let $P$ be a metric space, $X$ and $Y$ be complete metric spaces, $F:P\times X\rightrightarrows Y$ and $G:P\rightrightarrows X$ be given by \eqref{G2}.
Suppose that $\gph F_p$ is closed for all $p\in B_\eta(\bp)$.
The mapping $G$ has the Aubin property at $(\bp,\bx)$ with rate $l>0$, and $\eta$, $\de$ and $\mu$ if, for some $l'>0$, the mapping $F$ $l'-$recedes in $p$ uniformly in $x$ at $(\bp,\bx,\by)$ with $\eta$, $\de$ and $\mu$, and one of the following conditions holds true:
\begin{enumerate}
\item
there exists a $\ga>0$ such that
\begin{align*}%\label{C1-1}
\limsup_{\substack{u\to x,\,v\to y,\,(u,v)\in\gph F_p,\,(u,v)\ne (x,y)\\
d(u,\bx)<\de+l\mu,\,d(v,\by)<l'\mu}}
{\dfrac{d(y,\by)-d(v,\by)}{d_\ga((u,v),(x,y))}}\ge\frac{l'}{l}
\end{align*}
for all $p$, $x$ and $y$ satisfying
\begin{align}\label{P5.8-3}
p\in B_\eta(\bp),\;x\in B_{{\de+\mu}}(\bx)\setminus F_p\iv(\by),\;
y\in F(p,x){\cap B_{l'\mu}(\by)};
\end{align}

\item
$X$ and $Y$ are {Banach}, and
there exists a $\ga>0$ such that, with $N:=N^C$,
\begin{align}\label{P5.8-4}
d_\ga((0,-y^*),N_{\gph F_p}(x,y))\ge\frac{l'}{l}
\end{align}
for all $p$, $x$ and $y$ satisfying \eqref{P5.8-3}, and all $y^*\in Y^*$ satisfying \eqref{T1-1};

\item
$X$ and $Y$ are {Asplund}, and
there exist a $\ga>0$ and a $\tau\in]0,1[$ such that condition \eqref{P5.8-4} is satisfied with $N:=N^F$
for all $p$, $x$ and $y$ satisfying \eqref{P5.8-3}, and all $y^*\in Y^*$ satisfying~\eqref{T1-2};

\item
$X$ and $Y$ are {Banach}, and
%there exists an $\eta\in]0,+\infty]$ such that,
\begin{align}\label{P5.8-5}
d(0,D^*F_p(x,y)(B_{\eta}(y^*)))\ge\frac{l'}{l}
\end{align}
{with ${D^*:=D^*_C}$} for all $p$, $x$ and $y$ satisfying \eqref{P5.8-3}, and all $y^*\in Y^*$ satisfying \eqref{T1-1};

\item
$X$ and $Y$ are Asplund, and
{there exists}
%an $\eta\in]0,+\infty]$ and
a $\tau\in]0,1[$ such that condition \eqref{P5.8-5} is satisfied with $D^*:=D^*_F$
for all $p$, $x$ and $y$ satisfying \eqref{P5.8-3}, and all $y^*\in Y^*$ satisfying~\eqref{T1-2}.
\sloppy
\end{enumerate}
\end{proposition}

\if{
\AK{14/06/20.
Should some comments to this proposition be added as suggested by one of the reviewers in their last remark?}
}\fi
\begin{remark}
Conditions (i) in \cref{P5.8} can be seen as a quantitative version of \cite[Theorem~3.9]{Iof17.1}, while conditions (iv) and (v) improve \cite[Theorem~4.1]{Iof17.1} and \cite[Theorem~7.26]{Iof17}.
Conditions (ii) and (iii) are new. Note that these two conditions are weaker than conditions (iv) and (v), respectively.
\end{remark}

\addcontentsline{toc}{section}{References}

\nocite{*}
%%%%%%%%%%%%%%%%%%%%%%%%%%%%%%%%%%%%%%%%%%%%%%%
\bibliographystyle{jnsao}
\bibliography{IM}
%%%%%%%%%%%%%%%%%%%%%%%%%%%%%%%%%%%%%%%%%%%%%%%

%%%%%%%%%%%%%%%%%%%%%%%%%%%%%%%%%%%%%%%%%%%%%%%
\end{document}